\documentclass[a4paper]{article}
\UseRawInputEncoding

\usepackage{amsthm}

\newtheorem{thm}{Theorem}[section]
\newtheorem{lem}{Lemma}[section]

\newtheorem{cor}{Corollary}[section]

\newtheorem{prop}{Proposition}[section]

\usepackage[toc,page]{appendix}
\usepackage{url}
\usepackage{pgfplots}
\usepackage{tikz}
\usepackage{amsmath}
\usepackage{mathtools}
\usepackage[a4paper, left=3cm, right=3cm, top=4cm]{geometry}
\usepackage{amssymb}
\usepackage{amsfonts}
\usepackage{scrlayer-scrpage}
\usepackage{hyperref}
\usepackage{accents}
\usepackage{ulem}
\clearscrheadfoot
\newcommand{\R}{\mathbb{R}}

\newcommand{\Z}{\mathbb{Z}}

\def\XXint#1#2#3{{\setbox0=\hbox{$#1{#2#3}{\int}$ }
\vcenter{\hbox{$#2#3$ }}\kern-.6\wd0}}

\AtBeginDocument{
  \let\div\relax
  \DeclareMathOperator{\div}{div}
}

\usepackage{setspace}
\makeatletter
\newcommand{\MSonehalfspacing}{%
  \setstretch{1.44}
  \ifcase \@ptsize \relax 
    \setstretch {1.448}%
  \or 
    \setstretch {1.399}%
  \or 
    \setstretch {1.433}%
  \fi
}
\newcommand{\MSdoublespacing}{%
  \setstretch {1.92}
  \ifcase \@ptsize \relax 
    \setstretch {1.936}%
  \or 
    \setstretch {1.866}%
  \or 
    \setstretch {1.902}%
  \fi
}
\makeatother

\ohead{\pagemark}

\begin{document}

	\title{Existence, Uniqueness and Regularity of the Fractional Harmonic Gradient Flow in General Target Manifolds}
	
	\author{Jerome Wettstein}
\maketitle
\date{ }
\begin{abstract}
In this paper, we continue to study the fractional harmonic gradient flow on $S^1$ taking values in a general closed manifold $N \subset \R^n$, addressing global existence and uniqueness of solutions of energy class with sufficiently small energy, adding to the existing body of knowledge pertaining to the half-harmonic gradient flow and expanding upon our previous work in \cite{wettstein}. We extend the techniques by Struwe in \cite{struwe} and Rivi\`ere in \cite{riv} to the non-local framework analogous to \cite{wettstein} to derive uniqueness, employ commutator estimates as in \cite{daliopigati} for regularity and follow \cite{struwe} for a general existence result.

\end{abstract}

\medskip

\tableofcontents

	\section{Introduction}
	
	The goal of this paper is to expand upon the findings of the author's previous work in \cite{wettstein}, where the half-harmonic gradient flow with values in $S^{n-1}$ was studied. More precisely, the following result was proven:
	\begin{thm}
	\label{mainresultsection1sn}
		Let $u_0 \in H^{1/2}(S^1;S^{n-1})$ be any initial data. There exists $\varepsilon > 0$, such that if:
		$$\| ( - \Delta)^{1/4} u_0 \|_{L^2(S^1)} \leq \varepsilon,$$
		then there exists a unique energy class solution $u: \R_{+} \times S^1 \to S^{n-1} \subset \R^n$ of the weak fractional harmonic gradient flow:
		$$u_t + (-\Delta)^{1/2} u = u | d_{1/2} u |^2,$$
		satisfying $u(0, \cdot) = u_0$  in the sense $u(t,\cdot) \to u_0$ in $L^2$, as $t \to 0$. Moreover, the solution fulfills the energy decay estimate:
		$$\| (-\Delta)^{1/4} u(t) \|_{L^2 (S^1)} \leq \| (-\Delta)^{1/4} u_0 \|_{L^{2}(S^1)}.$$
		In fact, $u \in C^{\infty}(]0,\infty[ \times S^1)$ and for an appropriate subsequence $t_k \to \infty$, the sequence $u(t_k)$ converges weakly in $H^{1}(S^1)$ to a point.
	\end{thm}
	Let us pause here for a moment and briefly recall what (fractional) harmonic maps are: Harmonic maps are the critical points of the following, nowadays quite standard Dirichlet energy which is given for all maps $u: M \to N \subset \R^n$ in $H^{1}(M;N)$ by:
	$$E(u) := \frac{1}{2} \int_{M} g^{\alpha \beta}(x) \gamma_{ij}(u(x)) \frac{\partial u^{i}}{\partial x_{\alpha}}(x) \frac{\partial u^{j}}{\partial x_{\beta}}(x) dx,$$ 
	where $(M,g), (N, \gamma)$ smooth Riemannian manifolds, $u = (u^{1}, \ldots, u^{n})$ and employing Einstein's summation convention. In case $M = \Omega \subset \R^m$ and $N \subset \R^{n}$ are isometrically embedded in $\R^m$ and $\R^n$ and equipped with the Riemannian metrics induced by the standard scalar product, this reduces to:
	$$E(u) = \frac{1}{2} \int_{\Omega} | \nabla u |^2 dx$$
	One naturally is lead to the following extension: We say that a map $u: S^1 \to N \subset \R^n$ is weakly $1/2$-harmonic, if it is a critical point of the following energy:
	\begin{equation}
	\label{halfenergysect1}
		E_{1/2}(u) := \frac{1}{2} \int_{S^1} | (-\Delta)^{1/4} u |^2 dx,
	\end{equation}
	with respect to variations in the following set:
	$$H^{1/2}(S^1;N) := \big{\{} v \in H^{1/2}(S^1; \mathbb{R}^{n})\ \big{|}\ u(x) \in N, \text{ for a.e. } x \in S^1 \big{\}}$$
	Namely, the criticality condition means that for every $\Phi \in \dot{H}^{1/2}(S^1;\mathbb{R}^{n}) \cap L^{\infty}(S^1)$, in particular all smooth $\Phi \in C^{\infty}(S^1;\R^n)$, we have:
	\begin{equation}
		\frac{d}{dt} E_{1/2} \left( \pi( u + t \Phi ) \right) \Big{|}_{t = 0} = 0,
	\end{equation}
	where $\pi$ is the orthogonal closest-point projection to $N$, which is defined in a sufficiently small neighbourhood of $N$ and smooth due to $N$ being smooth. As we shall see, this condition is equivalent to:
	\begin{equation}
	\label{1/2harmonicbyorthogpojection}
		d\pi(u) (-\Delta)^{1/2} u = 0 \quad \text{ in } \mathcal{D}'(S^1),
	\end{equation}
	which is sometimes also stated informally in the following form, observing that $d\pi(x)$ is the orthogonal projection to $T_{x} N$ for every $x \in N$:
	$$(-\Delta)^{1/2} u \perp T_{u}N$$
	It is clear that, in order to study the regularity of $1/2$-harmonic maps, the first step lies in the reformulation of \eqref{1/2harmonicbyorthogpojection}. Naturally, corresponding definitions for $\R$ instead of $S^1$ are possible. Such equations were first studied in \cite{dalioriv} and questions regarding regularity, bubbling and general properties of such maps have been adressed in the literature, see \cite{daliocomment1}, \cite{schikorradaliocomment}, \cite{dalioschikorra}, \cite{daliocomment2}, \cite{daliopigati} and \cite{daliolaurainriv}. $1/2$-harmonic maps are for example linked to free-boundaries of minimal surfaces.\\
	
	In this paper, we will study the associated evolution problem with the energy \eqref{halfenergysect1} for arbitrary closed manifolds $N \subset \mathbb{R}^n$. We shall see that this equation could be phrased as:
	\begin{equation}
	\label{maineqsect1v1}
		u_{t} + (-\Delta)^{1/2} u = (Id - d\pi(u)) (-\Delta)^{1/2} u,
	\end{equation}
	or:
	\begin{equation}
	\label{maineqsect1v2}
		u_t + (-\Delta)^{1/2} u = d_{1/2} u \cdot d_{1/2} \left( d\pi^{\perp}(u) \right) + \div_{1/2} \left( \frac{A^{i}_{u}(du, du)(x,y)}{| x-y |^{1/2}} d\pi^{\perp}(u(y))_{ij} \right),
	\end{equation}
	for $u$ being a function assuming values a.e. in $N$ with appropriate initial condition $u(0) = u_0$ with values in $N$. Other types of reformulations are possible and will appear later on. Similar problems in the local setup have been studied in \cite{struwe1} and \cite{struwe2} for the evolution problems associated with the harmonic map equation and found existence in an appropriate sense. However, the case of weak solutions was resolved by \cite{riv} in the case of small energy by means of uniqueness and \cite{freire} later on in general. As in \cite{wettstein}, we will focus here on the case of small energy for solutions in the weakest sense and study a class of general solutions for arbitrary energy much like in \cite{struwe1}. More precisely, we shall prove:
	
	\begin{thm}
	\label{mainresultsection1}
		Let $u_0 \in H^{1/2}(S^1;N)$ be any initial datum and $N$ be any closed manifold. There exists $\varepsilon > 0$, such that if:
		$$\| ( - \Delta)^{1/4} u_0 \|_{L^2(S^1)} \leq \varepsilon,$$
		then there exists a unique energy class solution $u: \R_{+} \times S^1 \to N \subset \R^n$ of the weak fractional harmonic gradient flow \eqref{maineqsect1v1}, \eqref{maineqsect1v2} satisfying $u(0, \cdot) = u_0$  in the sense $u(t,\cdot) \to u_0$ in $L^2$, as $t \to 0$. Moreover, the solution fulfills the energy decay estimate:
		$$\| (-\Delta)^{1/4} u(t) \|_{L^2 (S^1)} \leq \| (-\Delta)^{1/4} u_0 \|_{L^{2}(S^1)}.$$
		In fact, $u \in C^{\infty}(]0,\infty[ \times S^1)$ and for an appropriate subsequence $t_k \to \infty$, the sequence $u(t_k)$ converges weakly in $H^{1}(S^1)$ to a point. Without the small energy assumption, a unique solution $u \in C^{\infty}(]0,T[ \times S^1) \cap H^{1}([0,T] \times S^1; N)$ with non-increasing energy exists up to some time $T$ that can be bounded from below by the initial energy $\| (-\Delta)^{1/4} u_0 \|_{L^{2}(S^1)}$.
	\end{thm}
	
	It would be interesting to study the behaviour of solutions to the half-harmonic gradient flow with initial datum with high energy and see what happens. In particular, it would be worth investigating blow-ups of the solution in finite time. If no blow-ups exist, then one may argue as in \cite{struwe1} to extend solutions to arbitrary times, i.e. global smooth existence would be proven for all initial data, with uniqueness of the solution among all that have non-increasing energy. 
	
	The key techniques used will be very similar to \cite{wettstein} and we refer to this paper and, in particular, the introduction there for some more details on the techniques used. To briefly summarise, existence is obtained along a quite standard argument involving the inverse function theorem in Banach spaces. The most interesting point in the argument involves a bootstrap argument based on commutator estimates from \cite{daliopigati} and regularity results as in \cite{hieber}. A crucial step is the investigation of the kind of fractional heat equations solved by the difference between a candidate for a solution of the half harmonic gradient flow and its projection onto $N$, which ultimately allows us to prove that the candidate $u$ indeed assumes values in $N$. Uniqueness follows similar to \cite{struwe}, using ideas and reformulations from \cite{mazoschi} based on \eqref{maineqsect1v2} and arguments based on \cite{riv} to treat the energy class case with small energy by some compensation phenomenon. In fact, the compensation that occurs is due to an anti-symmetric potential and based on estimates found in \cite{daliopigati}. Lastly, the convergence result is an immediate adaption of \cite{struwe1}, as has previously been seen in \cite{wettstein}.\\
	
	The paper is organised as follows: In Section 2, we discuss and introduce some of the key notions for our later arguments. Section 3 starts our investigation of the fractional harmonic gradient flow in the case of $N$ being a closed, orientable hypersurface. The formula we find is reminiscent of the one in \cite{mazoschi} and \cite{wettstein} and emphasises the increased technical difficulty of dealing with general $N$. In Section 4, we finally turn to arbitrary closed manifolds $N$, first investigating the different formulations of the fractional harmonic gradient flow in Section 4.1. Then, we prove uniqueness of solutions under varying assumptions in Sections 4.2 by following \cite{struwe1} and \cite{riv}. Next, in Section 4.3, we deal with local existence for smooth boundary data using ideas similar to \cite{hamilton} and use estimates as in \cite{struwe1} to deduce local existence and global existence for small initial energy as in \cite{wettstein}. Indeed, the differences in the proofs in \cite{wettstein} and the current paper are minor, as the technique relies on general properties of the non-linearity (quadratic growth in an appropriate sense, orthogonality to the tangent space of $N$, etc.). Convergence results as $t \to \infty$ are discussed in Section 4.4 and Section 4.5 studies the behaviour of the solution close to points of concentration of energy. The appendices complement the presentation.\\
	
	\textbf{Acknowledgements}
	I would like to thank my supervisor, Prof. Francesca Da Lio, for suggesting this problem, providing helpful insights throughout the process of working on this paper and their feedback on previous versions of this paper.

	\section{Preliminaries}
	
	Before we enter our discussion of the main result of this paper, we recall some notions from non-local analysis. In particular, we present the definition of the Triebel-Lizorkin spaces on the unit circle, give an equivalent characterisation under some technical assumptions as in \cite{schiwang} and define the fractional gradient and fractional divergence that will appear later on, together with some useful identities.
	
	\subsection{Fractional Laplacian and Triebel-Lizorkin Spaces}
	
	In this section, we recall the definition of the Triebel-Lizorkin spaces on the unit circle $S^1 \subset \R^2$ as well as some of the most relevant properties of the fractional Laplacian, at least for our purposes. The current presentation follows the one in \cite{schiwang} and \cite{schmeitrieb}.\\
	
	As already seen in \cite{wettstein}, we have a natural metric on $S^1$ stemming from the identification $S^1 \simeq \R / 2 \pi \Z$, providing a useful formula for the metric on the universal covering of $S^1$. The natural distance function given by:
	$$| x-y |^2 = | e^{ix} - e^{iy} |^2 = | e^{i(x-y)} - 1 |^2,$$
	which can be rewritten as:
	\begin{equation}
		| x-y | = 2 \left| \sin \left( \frac{x-y}{2} \right) \right|.
	\end{equation}
	We shall implicitly use this metric, whenever we are working over $S^1$, without emphasizing this fact further. Next, we define for any $f: S^1 \to \R$:
	$$\mathcal{D}_{s,q}(f)(x) := \left( \int_{S^1} \frac{| f(x) - f(y) |^q}{| x-y |^{sq}} \frac{dy}{| x-y |} \right)^{1/q},$$
	for all $1 \leq q < \infty$ and $0 < s < 1$. This results in the following definition as seen previously in \cite{schiwang}:
	$$\| f \|_{\dot{W}^{s,(p,q)}(S^1)} := \| \mathcal{D}_{s,q}(f)(x) \|_{L^{p}(S^1)},$$
	for every $1 \leq p \leq \infty$. If $p = q$, these spaces correspond to the usual homogeneous Gagliardo-Sobolev spaces $\dot{W}^{s,p}(S^1)$. The operator $\mathcal{D}_{s,q}$ and its main properties are studied in \cite{schiwang} and the references therein.\\
	
	As per usual, one denotes by $\mathcal{D}'(S^1)$ the collection of distributions on $S^1$ and sometimes denote, for notational convenience, by $\mathcal{D}(S^1)$ the space $C^{\infty}(S^1)$ of smooth functions (the collection of test functions). $\hat{f}(k)$ will always denote the $k$-th Fourier coefficient of $f$, for all $f \in \mathcal{D}'(S^1)$ and $k \in \mathbb{Z}$:
	$$\hat{f}(k) := \frac{1}{2\pi} \langle f, e^{-ikx} \rangle = \frac{1}{2\pi} f \left( e^{-ikx} \right), \quad \forall k \in \Z$$
	Completely analogous to the situation on $\R^n$, the Triebel-Lizorkin spaces for $S^1$, denoted by ${F}^{s}_{p,q}(S^1)$, are defined for all $s \in \R$, $p,q \in [1, \infty[$ by the following identity:
	$$F^{s}_{p,q}(S^1) := \big{\{} f \in \mathcal{D}'(S^1) \ \big{|}\ \| f \|_{F^{s}_{p,q}} < +\infty \big{\}}$$
	Here we employ the norm defined below, analogous to the construction of function spaces on $\R^n$:
	\begin{equation}
	\label{triebels1norm}
		\| f \|_{F^{s}_{p,q}} := \Bigg{\|} \Bigg{\|} \left( \sum_{k \in \mathbb{Z}} 2^{js} \varphi_{j}(k) \hat{f}(k) e^{ikx} \right)_{j \in \mathbb{N}} \Bigg{\|}_{l^{q}} \Bigg{\|}_{L^{p}(S^1)},
	\end{equation}
	for an appropriate partition of unity $(\varphi_{j})_{j \in \mathbb{N}}$ consisting of smooth, compactly supported functions on $\R$ satisfying:
	$$\operatorname{supp} \varphi_{0} \subset B_{2}(0), \quad \operatorname{supp} \varphi_{j} \subset \{ x \in \R\ |\ 2^{j-1} \leq | x | \leq 2^{j+1} \}, \forall j \geq 1$$
	as well as the boundedness property:
	$$\forall k \in \mathbb{N}: \sup_{j \in \mathbb{N}} 2^{jk} \| D^{k} \varphi_j \|_{L^{\infty}} \lesssim 1$$
	Such a family of functions can be easily constructed by the usual methods for Littlewood-Paley decompositions involving scalings. The Triebel-Lizorkin spaces on $S^1$, and more generally on the $n$-torus, possess a theory analogous to the classical case of function spaces on $\R^n$, see \cite{schmeitrieb}, Chapter 3. In particular, Sobolev embeddings continue to hold (\cite[Section 3.5.5]{schmeitrieb}), identifications with classical spaces such as $L^{p}(S^1)$ (\cite[Section 3.5.4]{schmeitrieb}) and duality results (\cite[Section 3.5.6]{schmeitrieb}). We shall use the properties of these spaces throughout this paper and shall refer to the given reference for details. The homogeneous spaces is now defined by omitting the Fourier coefficient of $0$th-order and adapting the notions accordingly.\\
	
	In our later considerations, it will be most convenient to be able to work with norms different from, but equivalent to \eqref{triebels1norm}. The reason lies in the technical nature of the norm \eqref{triebels1norm} which we shall not see explicitely emerge from the structure of the fractional gradient flow, but rather a different incarnation. More precisely, in \cite{schiwang}, the authors prove the following result:
	
	\begin{thm}[Theorem 1.4, \cite{schiwang}]
	\label{schiwangthm1.4}
		Let $s \in (0,1)$, $p,q \in ]1, \infty[$ and $f \in L^p(\R^n)$. Then:
		\begin{itemize}
			\item[(i)] We know $\dot{W}^{s, (p,q)}(\R^n) \subset \dot{F}^{s}_{p,q}(\R^n)$ together with:
			\begin{equation}
				\| f \|_{\dot{F}^{s}_{p,q}(\R^n)} \lesssim \| f \|_{\dot{W}^{s, (p,q)}(\R^n)}
			\end{equation}
			
			\item[(ii)] If $p > \frac{nq}{n + sq}$, then we also have the converse inclusion together with:
			\begin{equation}
			\label{secondpartofthm2.1}
				\| f \|_{\dot{W}^{s, (p,q)}(\R^n)} \lesssim \| f \|_{\dot{F}^{s}_{p,q}(\R^n)}
			\end{equation}
		\end{itemize}
		The constants depend on $s, p, q, n$.
	\end{thm}
	
	We mention that the spaces introduced before easily generalize to $\mathbb{R}^n$. Additionally, we refer to \cite{wettstein} for further details on references for this result.\\
	
	As seen in \cite{schiwang} and by using the properties in \cite{schmeitrieb}, \cite{triebel} for periodic functions, we can similarily discover the following equivalence with Triebel-Lizorkin spaces for all $1 < q < \infty$ and $1 < p < \infty$:
	\begin{equation}
		\dot{W}^{s, (p,q)}(S^1) = \dot{F}^{s}_{p,q}(S^1),
	\end{equation}
	with equivalence of the corresponding seminorms, provided $p > \frac{q}{1+sq}$. For a proof of a part of the result above in the case $S^1$, we refer to the Appendix in \cite{wettstein}. If $s = 1/2$ and $q = 2$, then $p > 1$ is the requirement in Theorem \ref{schiwangthm1.4} for the equality of $\dot{F}^{1/2}_{p,2}$ and $\dot{W}^{1/2, (p,2)}$ to hold. Moreover, if $q = 2$, an ubiquitous situation throughout this paper, the result surely applies for all $p \geq 2$. It should be observed that while $\dot{F}^{s}_{p,2}(S^1) \subset \dot{W}^{s,p}(S^1) = \dot{W}^{s,(p,p)}(S^1)$ for $p \geq 2$, there does not hold equality except for $p =2$. The reader is reminded of the difference between the Bessel potential spaces and the Gagliardo-Sobolev spaces, which is more or less the underlying statement of this inclusion. We will generally omit mentioning the domain, if it is clear from the context.\\
	
	On $S^1$, the fractional $s$-Laplacian is defined as a Fourier multiplier operating on Fourier series:
	$$\widehat{(-\Delta)^{s} f}(k) = | k |^{2s} \hat{f}(k),$$
	for every $k \in \mathbb{Z}$ and all $0 < s$. On the other hand, for $0 < s \leq 1$, this operator can be defined by a singular integral as well:
	$$(-\Delta)^{s} f(x) = C(s) \cdot P.V. \int_{S^{1}} \frac{f(x) - f(y)}{| x-y |^{1+2s}} dy,$$
	where $C(s) > 0$ denotes some constant depending on $s$. By the Fourier multiplier properties, fractional Laplacians interact in a natural way with Triebel-Lizorkin spaces $\dot{F}^{s}_{p,q}(S^1)$, as is usual for this type of function spaces. This means that it induces an isomorphism:
	$$(-\Delta)^{s}: \dot{F}^{t+2s}_{p,q} \to \dot{F}^{t}_{p,q},$$
	for all $p,q \in (1, \infty)$ and $t, t+2s \in \R$, see \cite[Section 3.6.3]{schmeitrieb} and the proof of the analogous statement in the case $\R^n$.\\
	
	In analogy, the $s$-Laplacian can be defined on $\R$ as a Fourier multiplier using the Fourier transform rather than the Fourier series and leads again to an object which can also be characterised by a similar principal value. We omit the details, as the formulas are virtually the same as for the circle.

	\subsection{Fractional Gradients and Divergences}
	
	For our later use, we summarise and collect some of the ideas in \cite{mazoschi}. Namely, we are most interested in the fractional gradient and fractional divergence and we recapitulate some of the notions, as was already done in \cite{wettstein}.\\
	
	We denote by $\mathcal{M}_{od}(\R \times \R)$ the set of all measurable functions $f: \R \times \R \to \R$ with respect to the measure $\frac{dx dy}{ | x-y |}$. One can make this definition equally well on $S^1$ by exchanging the domain $\R$ for the $S^1$ and using the metric previously mentioned. Whenever a definition/property applies for both domains, we shall sometimes denote this space by $\mathcal{M}_{od}$.\\
	
	For a measurable function $f: \R \to \R$ or $f: S^1 \to \R$, we define the \textit{fractional $s$-gradient} as follows:
	$$d_s f(x,y) = \frac{f(x) - f(y)}{| x-y |^s} \in \mathcal{M}_{od},$$
	for all $0 \leq s < 1$. The corresponding $s$-divergence is then introduced by means of duality. It should be clear, but is often useful to know that:
	$$d_{s} f(y,x) = - d_{s} f(x,y)$$
	As stated above, by duality, for $F \in \mathcal{M}_{od}(\R \times \R)$ or $F \in \mathcal{M}_{od}(S^1 \times S^1)$, we are consequently able to define for every $\varphi$ smooth and compactly supported on $\R$ or just smooth on $S^1$ in the latter case:
	$$\div_{s} F (\varphi) := \int \int F(x,y) d_{s} \varphi(x,y) \frac{dx dy}{| x-y |}$$
	This expression is hence defined merely in a distributional sense, i.e. by its duality relation with $d_{s}$. For later use, we generally denote for $F, G \in \mathcal{M}_{od}$ over $\R$ or $S^1$:
	$$F \cdot G (x) := \int F(x,y) G(x,y) \frac{dy}{| x-y |}$$
	As an obvious special case, if $F = G$ we also write:
	$$F \cdot F(x) = | F |^{2}(x) \Rightarrow | F |(x) := \sqrt{F \cdot F(x)}$$
	One should immediately notice the relationship between the previously defined norms on $W^{s,(p,q)}(S^1)$. Indeed, we have:
	$$\| | d_{s} f | \|_{L^{p}(S^1)} = \| f \|_{\dot{W}^{s,(p,2)}(S^1)}.$$
	This provides a powerful characterisation of Triebel-Lizorkin spaces $\dot{F}^{s}_{p,q}(S^1)$ in terms of the fractional gradients $d_{s}$, under certain special technical conditions on $s,p,q$.\\
	
	It is also possible to prove up to constants which we shall ignore, as they have no effect on the results:
	$$(- \Delta)^{s} f = \div_{s} d_s f,$$
	which is particularily useful for the weak formulation of PDEs involving non-local operators. This equation is to be understood in the following sense:
	$$C_s \int d_s f \cdot d_s g (x) dx = \int (-\Delta)^{s} f \cdot g dx = \int (-\Delta)^{s/2} f \cdot (-\Delta)^{s/2} g dx,$$
	for the domains $S^1$ and $\R$. Lastly, the following identity, sometimes referred to as fractional Leibniz' rule, is often useful:
	\begin{equation}
	\label{fractionalleibnizruleintrod}
		d_{s} \left( fg \right)(x,y) = d_{s} f(x,y) g(x) + f(y) d_{s} g(x,y).
	\end{equation}
	This identity can be verified by directly inserting the definition. Another type of Leibniz rule is summarised in the following formula:
	\begin{equation}
		(-\Delta)^{1/2} (fg) = (-\Delta)^{1/2} f \cdot g + f (-\Delta)^{1/2} g  - d_{1/2} f \cdot d_{1/2} g,
	\end{equation}
	which again can be verified by directly inserting definitions. This formula also accounts for the commutator behaviour. Therefore, the fractional gradient may be used to account for the error in the Leibniz rule for the fractional Laplacian and specifies the order of the error.\\
	
	In general, one defines $L^{p}_{od}(S^1 \times S^1)$ or $L^{p}_{od}(\R \times \R)$ as the set of all measurable functions, such that the following norm is finite:
	$$\| F \|_{L^{p}_{od}} := \left( \int \int | F(x,y) |^{p} \frac{dy dx}{| x-y |} \right)^{1/p},$$
	for $1 \leq p < \infty$. Obviously, $L^{\infty}_{od}(S^1 \times S^1)$ and $L^{\infty}_{od}(\R \times \R)$ are then to be introduced in the usual manner. These spaces are, in some sense, related to the spaces $W^{s,(p,q)}$.\\

	\section{The Fractional Harmonic Gradient Flow with Values in an orientable Hypersurface}
	
	Before we turn our attention to the case of a general target manifold, we dedicate some time to the uniqueness under improved regularity for the special case of an embedded hypersurface which is orientable and closed. This case exhibits similar properties as in the case of the $n-1$-sphere while essentially containing all features encountered in the general case. Moreover, the harmonic map equation possesses a slightly simpler form than in the general case, rendering this special case more tractable. However, the main reason to consider this special case lies in the emergence of all phenomena which we shall encounter in the general case, in particular the inclusion of a fractional divergence term, and thus providing a toy example which will simplify our treatment of the case of a general target manifold.\\
	
	Indeed, one of the main differences between the sphere $S^{n-1}$ and $N$ a hypersurface will be that the latter is described by a non-local PDE for the fractional harmonic flow which involves a fractional divergence. The techniques used here can then be rather easily adapted to the more general framework, as all estimates used are in some sense independent of the restrictions on $N$. The remaining properties contained in Theorem \ref{mainresultsection1} shall be proven in the next section for all $N$ at the same time.
	
	\subsection{The Euler-Lagrange Equation of the Half-Harmonic Map}
	Let us consider $N \subset \mathbb{R}^{n}$ a closed hypersurface, i.e. an orientable, compact submanifold of dimension $n-1$ without boundary. An important example is of course $N = S^{n-1}$. Under these circumstances, there exists a smooth unit normal field $\nu$ over $N$ which, using the tubular neighbourhood theorem and some cut-off-function, can be extended to a smooth vector field $\tilde{\nu}$ on all of $\R^{n}$, such that $\nu = \tilde{\nu}$ on $N$ and that $\nu$ is a unit vectorfield in a neighbourhood of $N$.\\
	
	Our goal is now to rewrite the $1/2$-harmonic map equation for maps with values in $N$. Following the computations in \cite{wettstein}, one may find along the same lines a formulation for the $1/2$-harmonic gradient flow. First, we recall from the introduction that a map $u: S^1 \to N \subset \R^n$ is called $1/2$-harmonic, if it is a critical point of the fractional $1/2$-Dirichlet energy:
	$$E(u) := \frac{1}{2} \int_{S^1} |(-\Delta)^{1/4} u |^2 dx,$$
	with respect to variations in $H^{1/2}(S^1;N)$. By compactness of $N$, we know that any element in this function space is almost everywhere bounded:
	$$H^{1/2}(S^1;N) \subset L^{\infty}(S^1)$$
	Let now $u(t)$ be a variation in the set of functions introduced in the introduction, such that $u(0) = u$ is a critical point of the fractional energy $E$. We may use the tubular neighbourhood theorem to construct $u(t) = \pi(u + t \varphi)$ for some $\varphi \in C^{\infty}(S^1)$. Here, we used $\pi$ to denote the projection onto $N$ which is well-defined and smooth on a sufficently small neighbourhood and thus for $t$ small enough. This means:
	$$u'(0) := \frac{d}{dt} u(t) \big{|}_{t=0} = d\pi(u) \varphi$$
	Then, we have for a critical point $u$ of $E$:
	\begin{align}
		0	&= \frac{d}{dt} E(u(t)) \big{|}_{t=0} \notag \\
			&= \frac{d}{dt} \left( \frac{1}{2} \int_{S^1} |(-\Delta)^{1/4} u(t) |^2 dx \right) \Big{|}_{t=0} \notag \\
			&= \int_{S^1} (-\Delta)^{1/4} u \cdot (-\Delta)^{1/4} u'(0) dx \notag \\
			&= \int_{S^1} (-\Delta)^{1/2} u \cdot d\pi(u) \varphi dx, 
	\end{align}
	which, thanks to $d\pi$ being an orthogonal projection onto the tangent space $T_{u} N$, can be rephrased as:
	\begin{equation}
		(-\Delta)^{1/2} u \perp T_u N
	\end{equation}
	This computation holds even for $N$ which are merely closed, there is no need to assume for example codimension $1$. This condition becomes however useful, if we would like to find an explicit formula for the harmonic map equation like in the local case, see \cite{struwe1}, or as we have seen for $N = S^{n-1}$ in \cite{mazoschi} or \cite{wettstein}. Namely, we know that:
	$$(-\Delta)^{1/2} u = \lambda \cdot \nu(u),$$
	where $\lambda$ is a scalar and depends on the point on $S^1$ inserted into $u$. Equivalently:
	$$\lambda = (-\Delta)^{1/2} u \cdot \nu(u)$$
	Take $\psi$ to be any smooth, scalar-valued function on $S^1$. We may then compute, by using the fractional Leibniz rule \eqref{fractionalleibnizruleintrod}:
	\begin{align}
		\int_{S^1} \lambda \cdot \psi dx	&= \int_{S^1} (-\Delta)^{1/2} u \cdot \psi \nu(u) dx \notag \\
								&= \int_{S^1} d_{1/2} u(x,y) \cdot d_{1/2} \left( \psi \nu(u) \right)(x,y) \frac{dx dy}{| x-y |} \notag \\
								&= \int_{S^1} d_{1/2} u(x,y) \cdot \left( d_{1/2} \left( \nu \circ u \right)(x,y) \psi(x) + d_{1/2} \psi(x,y) \nu(u(y)) \right) \frac{dx dy}{| x-y |} \notag \\
								&= \int_{S^1} d_{1/2} u(x,y) \cdot \left( d_{1/2} \left( \nu \circ u \right)(x,y) \psi(x) + d_{1/2} \psi(x,y) \frac{ \nu(u(x)) + \nu(u(y))}{2} \right) \frac{dx dy}{| x-y |},
	\end{align}
	where we used a change of variables (i.e. exchanging $x$ with $y$ and vice versa) to justify the last equation. Let us observe that we therefore have:
	\begin{equation}
	\label{factorlambda}
		\lambda = d_{1/2} u \cdot d_{1/2} \left( \nu \circ u \right) + \div_{1/2} \left( d_{1/2} u(x,y) \frac{\nu(u(x)) + \nu(u(y))}{2} \right)
	\end{equation}
	We emphasise that the operator $\div_{1/2}$ is defined precisely as the dual of the fractional gradient $d_{1/2}$, so that the identity in fact holds true. Let us observe that the first summand is actually hiding a quadratic structure similar to the one in the case $N = S^{n-1}$ (see \cite{wettstein}) or the local case. Namely, we observe that by the fundamental theorem of calculus:
	\begin{align}
		d_{1/2} \left( \nu \circ u \right)(x,y)	&= \frac{\nu(u(x)) - \nu(u(y))}{| x-y |^{1/2}} \notag \\
									&= \frac{1}{| x-y |^{1/2}} \int_{0}^{1} d\tilde{\nu} \left( u(y) + s \left( u(x) - u(y) \right) \right) \left( u(x) - u(y) \right) ds \notag \\
									&= \int_{0}^{1} d\tilde{\nu} \left( u(y) + s \left( u(x) - u(y) \right) \right) ds \cdot d_{1/2} u(x,y) \notag \\
									&=: \tilde{A_{u}}(x,y) d_{1/2} u(x,y),
	\end{align}
	where we notice the similarity of $\tilde{A_{u}}$ in a certain sense with the term appearing in the local case. We notice that $\tilde{A_{u}}$ is bounded, therefore giving the estimate:
	$$\big{|} d_{1/2} u \cdot \tilde{A_{u}} d_{1/2} u(x,y) \big{|} \leq \| \tilde{A_{u}} \|_{L^{\infty}} | d_{1/2} u |^{2}$$
	
	\subsection{Toy Example: Uniqueness under Improved Regularity}
	
	We now turn to the gradient flow associated with the fractional harmonic map with values in $N \subset \R^{n}$. Therefore, as in the case $N = S^{n-1}$ treated in \cite{wettstein}, let us assume that $u, v$ are two solutions to the fractional gradient flow taking a.e. values in the closed, orientable hypersurface $N \subset \R^n$ and we suppose the following regularity conditions hold:
	\begin{equation}
	\label{regasshyper}
		u,v \in L^{\infty}(\R_{+}; H^{1/2}(S^1)) \cap L^{2}_{loc}(\R_{+}; H^{1}(S^{1})); \quad u_t, v_t \in L^{2}(\R_{+}; L^{2}(S^{1}))),
	\end{equation}
	In addition, they satisfy the gradient flow associated with the $1/2$-harmonic map as described below (see the discussion in \cite{wettstein} and the previous subsection for a justification of this equation):
	\begin{equation}
	\label{gradflowuvhyper}
		w_t + (-\Delta)^{1/2} w = d_{1/2} w \cdot \tilde{A_{w}} d_{1/2} w \cdot \nu(w) + \div_{1/2} \left( d_{1/2} w(x,y) \frac{\nu(w(x)) + \nu(w(y))}{2} \right) \cdot \nu(w),
	\end{equation}
	for both $w = u$ and $w = v$, together with the boundary condition $u(0, \cdot) = v(0, \cdot) = u_{0} \in H^{1/2}(S^1;N)$. It is intuitively clear that the same arguments as in the proof of Theorem 3.2 in \cite{wettstein} should be applicable to the current situation to deduce an analogous uniqueness result, as long as we assume the same kind of regularity for the solution as there. Indeed, we shall prove:
	
	\begin{thm}
	\label{uniquehyper}
		If $u,v$ both solve \eqref{gradflowuvhyper} with the same initial datum $u_0 \in H^{1/2}(S^1;N)$ and we assume that:
		$$\| (-\Delta)^{1/4} u(t) \|_{L^{2}(S^1)}, \| (-\Delta)^{1/4} v(t) \|_{L^{2}(S^1)} \leq \| (-\Delta)^{1/4} u_{0} \|_{L^{2}(S^1)}, \quad \forall t \in \R_{+},$$
		then we have:
		$$u = v$$
	\end{thm}
	
	The proof is largely the same as the one for $N = S^{n-1}$, cf. \cite{wettstein}. The changes mostly consist of finding suitable decompositions of the different contributions for the situation at hand, in particular the divergence term. We therefore focus on providing the key estimates needed for the proof and refer to our previous work for the remaining details:
	
	\begin{proof}
		The main idea is to study the non-local PDE solved by the difference between $u$ and $v$. Therefore, we are led to define:
		$$w := u - v,$$
		and observe that:
		$$w(0, \cdot) = u(0, \cdot) - v(0, \cdot) = 0.$$
		We find that $w$ solves the following PDE by linearity of derivatives and the fractional Laplacians:
		\begin{align}
		\label{decompr1r2forhyper}
			w_{t} + (-\Delta)^{1/2} w	&= u_{t} + (-\Delta)^{1/2} u - v_{t} - (-\Delta)^{1/2} v \notag \\
								&= R_1 + R_2,
		\end{align}
		where:
		\begin{align}
		\label{remainder1}
			R_{1}	&:= d_{1/2} u \cdot \tilde{A_{u}} d_{1/2} u \cdot \nu(u) - d_{1/2} v \cdot \tilde{A_{v}} d_{1/2} v \cdot \nu(v) \\
		\label{remainder2}
			R_{2}	&:= \div_{1/2} \left( d_{1/2} u \frac{\nu(u(x)) + \nu(u(y))}{2} \right) \nu(u) - \div_{1/2} \left( d_{1/2} v \frac{\nu(v(x)) + \nu(v(y))}{2} \right) \nu(v)
		\end{align}
		Naturally, we would like to estimate $R_1$ and $R_2$ in a similar way as we have done for the $n-1$-sphere. We treat both contributions individually:\\
		
		Let us begin with $R_1$. We observe that by using the fundamental theorem of calculus, we can deal with the fractional gradients of $\nu(u), \nu(v)$, i.e. the additional term $\tilde{A_{u}}, \tilde{A_{v}}$. Namely, we have:
		\begin{align}
		\label{decompr1hyper}
			R_{1}	&= d_{1/2} u \cdot \tilde{A_{u}} d_{1/2} u - d_{1/2} v \cdot \tilde{A_{v}} d_{1/2} v \notag \\
					&= \left( d_{1/2} u \cdot \tilde{A_{u}} d_{1/2} u - d_{1/2} v \cdot \tilde{A_{u}} d_{1/2} v \right) + d_{1/2} v \cdot \left( \tilde{A_{u}} - \tilde{A_{v}} \right) d_{1/2} v \notag \\
					&=: R_{1,1} + R_{1,2}
		\end{align}
		For $R_{1,1}$, we proceed by using the fundamental theorem of calculus:
		\begin{align}
			R_{1,1}	&= \int_{0}^{1} \frac{d}{ds} \left( d_{1/2} \left( (1-s)v + s u \right) \cdot \tilde{A_{u}} d_{1/2} \left( (1-s)v + s u \right) \right) ds \notag \\
					&= \int_{0}^{1} \frac{d}{ds} \left( \int_{S^1} d_{1/2} \left( (1-s)v + s u \right)(x,y) \cdot \tilde{A_{u}}(x,y) d_{1/2} \left( (1-s)v + s u \right)(x,y) \frac{dy}{| x-y |} \right) ds \notag \\
					&\lesssim \int_{0}^{1} \int_{S^1} | d_{1/2} \left( (1-s)v + s u \right)(x,y) | \cdot | \tilde{A_{u}}(x,y) | \cdot | d_{1/2} \left( u - v \right)(x,y) | \frac{dy}{| x-y |} ds \notag \\
					&\lesssim \int_{0}^{1} \left| d_{1/2} \left( v + s(u-v) \right) \right|(x) \left| d_{1/2} w \right| (x) ds \notag \\
					&\lesssim \left( | d_{1/2} u |(x) + | d_{1/2} v |(x) \right) \cdot | d_{1/2} w |(x)
		\end{align}
		This contribution can be dealt with just as in the case $N = S^{n-1}$ after we test against $w$ itself, see \cite{wettstein}. Concerning $R_{1,2}$, we may use:
		\begin{align}
			&\tilde{A_{u}}(x,y) - \tilde{A_{v}}(x,y) \notag \\	
								&= \int_{0}^{1} d\tilde{\nu} \left( u(y) + s \left( u(x) - u(y) \right) \right) ds - \int_{0}^{1} d\tilde{\nu} \left( v(y) + s \left( v(x) - v(y) \right) \right) ds \notag \\
								&= \int_{0}^{1} \frac{d}{dt} \left( \int_{0}^{1} d\tilde{\nu} \left( v(y) + t w(y) + s \left( (v(x) + t w(x) - v(y) - t w(y) \right) \right) ds \right) dt \notag \\
								&= \int_{0}^{1} \int_{0}^{1} d(d\tilde{\nu}) \left( (1-s)(v(y) + t w(y)) + s (v(x) + t w(x)) \right) \cdot \left( (1-s) w(y) + s w(x) \right) ds dt
		\end{align}
		Now we may obtain an estimate as in the case $N = S^{n-1}$ by testing against $w$ and:
		$$w_{i}(y) w_{j}(x) \leq \frac{w_{i}(y)^2 + w_{j}(x)^2}{2} \leq \frac{| w |^2(y) + | w |^2(x)}{2},$$
		which enables us, together with the symmetry of the emerging integrals in $x,y$, to arrive at an estimate reminiscent of the first term on the right hand side of equation (20) in \cite{wettstein}.\\
		
		It remains to consider $R_2$. This term does not possess an immediate analogue in the case $N = S^{n-1}$, thus it seems to require some additional work. However, similar ideas (together with using the duality definition of $\div_{1/2}$ and Leibniz' rule for fractional gradients as in the previous section) may be employed to reach a suitable estimates. Let us first notice that testing against $w$ yields for $\nu_{u}(x,y) := (\nu(u(x)) + \nu(u(y)))/2$:
		\begin{align}
		\label{decompr2equation}
			\int_{S^1} R_2 w dx	&= \int_{S^1} \int_{S^1} d_{1/2} u d_{1/2} \left( \nu(u) w \right) \nu_{u}(x,y) - d_{1/2} v d_{1/2} \left( \nu(v) w \right) \nu_{v}(x,y) \frac{dx dy}{| x-y |}
		\end{align}
		To conclude, we have to estimate this expression appropriately. We see:
		\begin{align}
		\label{decompr2hyper}
			d_{1/2} u 	&d_{1/2} \left( \nu(u) w \right) \nu_{u}(x,y) - d_{1/2} v d_{1/2} \left( \nu(v) w \right) \nu_{v}(x,y) \notag \\
					&= d_{1/2} u d_{1/2} \left( \nu(u) w \right) \nu_{u}(x,y) - d_{1/2} u d_{1/2} \left( \nu(v) w \right) \nu_{u}(x,y) \notag \\
					&+ d_{1/2} u d_{1/2} \left( \nu(v) w \right) \nu_{u}(x,y) - d_{1/2} v d_{1/2} \left( \nu(v) w \right) \nu_{u}(x,y) \notag \\
					&+ d_{1/2} v d_{1/2} \left( \nu(v) w \right) \nu_{u}(x,y) - d_{1/2} v d_{1/2} \left( \nu(v) w \right) \nu_{v}(x,y) \notag \\
					&=: R_{2,1} + R_{2,2} + R_{2,3}
		\end{align}
		More precisely, we have that:
		$$R_{2,1} = d_{1/2} u d_{1/2} \left( (\nu(u) - \nu(v)) w \right) \nu_{u}(x,y),$$
		and:
		$$R_{2,2} = d_{1/2} w d_{1/2} \left( \nu(v) w \right) \nu_{u}(x,y),$$
		and finally:
		$$R_{2,3} = d_{1/2} v d_{1/2} \left( \nu(v) w \right) \left( \nu_{u}(x,y) - \nu_{v}(x,y) \right)$$
		We shall use the fractional Leibniz' rule as already seen in Section 2, see \eqref{fractionalleibnizruleintrod}:
		$$d_{1/2} \left( f w \right)(x,y) = f(x) \frac{w(x) - w(y)}{| x-y |^{1/2}} + w(y) \frac{f(x) - f(y)}{| x-y |^{1/2}},$$
		which leads us to:
		$$d_{1/2} \left( ( \nu(u) - \nu(v) ) w \right)(x,y) = \left( \nu(u(x)) - \nu(v(x)) \right) d_{1/2}w(x,y) + d_{1/2} \left( \nu(u) - \nu(v) \right) w(y)$$
		So we may estimate $R_{2,1}$ for example as follows:
		\begin{align}
			\left| \int_{S^{1}} R_{2,1} dx \right|	&\leq \int_{S^1} | d_{1/2} u |(x) | d_{1/2} w |(x) | \nu(u(x)) - \nu(v(x)) | dx \notag \\
										&+ \int_{S^1} | d_{1/2} u |(y) \left| d_{1/2} \left( \nu(u) - \nu(v) \right) \right|(y) | w | dy \notag \\
										&\lesssim \int_{S^1} | d_{1/2} u |(x) | d_{1/2} w |(x) | w | dx,
		\end{align}
		by using again the smoothness of $\tilde{\nu}$ and renaming the variable of integration. This is a term that can be treated as before.\\
		
		The remaining summands $R_{2,2}$ and $R_{2,3}$ are a bit more delicate to deal with, mainly because we have to use another kind of estimate in our proceedings. Let us start with $R_{2,2}$ and observe the following: We shall denote by $\pi$ the smooth closest point projection defined in a neighbourhood of $N$ and extended to all of $\R^n$ by using a cut-off function (compare with the construction in the next section for some more details). We have for all $x \in S^1$, using Einstein's summation convention and $u = (u_1, \ldots, u_{n})$ and similarily for $v$ and $w$:
		\begin{align}
			\nu(v(x))_{i} w_{i}(x)			&= \nu(v(x))_{i}(u_{i}(x) - v_{i}(x)) \notag \\
									&= \nu(v(x))_{i} \left( \partial_{j} \pi(v(x))_{i} (u_{j}(x) - v_{j}(x)) \right) \notag \\
									&+ \nu(v(x))_{i} \left( \int_{0}^{1} \int_{0}^{1} t \partial_{kj} \pi(v(x) + ts w(x))_{i} (u_{j}(x) - v_{j}(x)) (u_{k}(x) - v_{k}(x)) ds dt \right) \notag \\
									&= \nu(v(x))_{i} \left( \int_{0}^{1} \int_{0}^{1} t \partial_{kj} \pi(v(x) + ts w(x))_{i} (u_{j}(x) - v_{j}(x)) (u_{k}(x) - v_{k}(x)) ds dt \right) \notag \\
									&=: D_{jk}^{u,v}(x) (u_{j}(x) - v_{j}(x)) (u_{k}(x) - v_{k}(x)) = D_{j,k}^{u,v}(x) w_{j}(x) w_{k}(x)
		\end{align}
		by using Taylor expansion around the point $v(x)$. Observe that we exploited the fact that $d\pi(v(x))$ maps to the tangent space $T_{v(x)} N$ of $N$ at $v(x)$, which immediately shows:
		$$\nu(v(x)) d\pi(v(x)) w(x) = 0,$$
		as $\nu(v(x))$ is orthogonal to the projected vector $d\pi(v(x)) w(x)$. If we insert this into $R_{2,2}$ instead of $\nu(v) w$, we see by using $d_0 w = d_0 u - d_0 v$, boundedness of $u,v$ and therefore also of $w$, and estimating $D_{jk}^{uv}$ and its fractional gradient using the smoothness of $\pi$:
		\begin{equation}
			\left| \int_{S^1} R_{2,2} dx \right| \lesssim \int_{S^1} | d_{1/2} w |(x) | w | \left( | d_{1/2} u |(x) + | d_{1/2} v |(x) \right) dx
		\end{equation}
		In fact, we used implicitly a decomposition of the following form:
		\begin{align}
			d_{1/2} \left( \nu(v) w \right)(x,y)	&= d_{1/2} \left( D_{j,k}^{u,v} w^{j} w^{k} \right)(x,y) \notag \\
									&= d_{1/2} D_{j,k}^{u,v}(x,y) \cdot w^{j}(x) w^{k}(x) \notag \\
									&+ D_{j,k}^{u,v}(y) \cdot d_{1/2} w^{j}(x,y) \cdot w^{k}(x) \notag \\
									&+ D_{j,k}^{u,v}(y) w^{j}(y) \cdot d_{1/2}w^{k}(x,y),
		\end{align}
		which can then be dealt with similar to the term $R_1$ and $R_{1,2}$. Note that $| d_{1/2} w |(x)$ and $| d_{1/2} D_{j,k}^{u,v} |(x)$ can both be bounded by $| d_{1/2} u |(x) + | d_{1/2} v |(x)$ and that $\| w \|_{L^{\infty}} < +\infty$, as $N$ is compact and $u,v$ both assume values a.e. in $N$. Similarily, it is clear that $D_{j,k}^{u,v}$ is bounded, due to the regularity of the extended closest-point-projection $\pi$.\\
		
		To arrive at a similar estimate for $R_{2,3}$, we notice the following by using the formula above for $\nu(v) w$ and exchanging the labels $x,y$ at some point:
		\begin{equation}
			\left| \int_{S^1} R_{2,3} dx \right| \lesssim \int_{S^1} (| d_{1/2} u |(x) + | d_{1/2} v |(x)) | d_{1/2} w |(x) | w | + (| d_{1/2} u |(x) + | d_{1/2} v |(x)) ^2(x) | w |^2 dx
		\end{equation}
		Comparing this with the estimates in the case $N = S^{n-1}$, we see that we may now proceed as in the proof there, since the main estimate in equation (23) of \cite{wettstein} can now be generalized and the remainder of the proof is of general nature and does not rely on any particular structure of $S^{n-1}$. As a result, this concludes our proof.
	\end{proof}
	
	Naturally, the proof of uniqueness also continues to work for a variety of non-linearities which are in some sense quadratic in the fractional gradient and sufficiently smooth by precisely the same arguments. Moreover, as we have seen in the case of general hypersurfaces, certain perturbation terms are allowed to appear without obstructing the argument, namely some kinds of divergence terms. Therefore, it is expected that analogous results hold for fractional harmonic flows with values in arbitrary closed smooth manifolds by means virtue of a suitable quadratic structure similar to the local case, where it is intimately connected to the curvature of the manifold. In fact, there is an explicit formula for the half-harmonic map given in \cite{mazoschi} which we shall exploit in the next section. There, we shall also present the missing argument for the improvement in regularity needed to conclude the proof of uniqueness in the energy class for solutions with small $1/2$-energy.\\
	
	To conclude this section, let us mention the following: The choice of equation \eqref{gradflowuvhyper} is natural for orientable hypersurfaces, due to the existence of a unit normal vector field $\nu$. Nevertheless, one could omit the use of such a vector field by using that the non-linearity in the $1/2$-harmonic map equation could be phrased as:
	\begin{align}
	\label{alternativeversionofequationhyper}
		&(-\Delta)^{1/2} u(x)	\notag \\
		&=\left( Id - d\pi(u)(x) \right) (-\Delta)^{1/2} u(x) \notag \\
		&= P.V. \int_{S^1} \frac{u(x) - u(y)}{| x-y |^2} dy - d\pi(u)(x) P.V. \int_{S^1} \frac{u(x) - u(y)}{| x-y |^2} dy \notag \\
		&= P.V. \int_{S^1} \frac{u(x) - u(y) - d\pi(u)(x) \left( u(x) - u(y) \right)}{| x-y |^2} dy \notag \\
		&= P.V. \int_{S^1} \int_{0}^{1} \int_{0}^{1} (s-1) d(d\pi)( u(x) + (st-t) (u(x) - u(y))) ds dt d_{1/2} u(x,y) d_{1/2} u(x,y) \frac{dy}{| x-y |} \notag \\
		&=: \int_{S^1} P(x,y) d_{1/2} u(x,y) d_{1/2} u(x,y) \frac{dy}{| x-y |},
	\end{align}
	where we denote again by $\pi$ also an extension by cut-off of the closest point projection and use $\pi(u(x)) = u(x)$ for $u$ being a solution to the $1/2$-harmonic map equation in $N$. The map $P$ is defined qppropriately and one sees that it is clearly bounded. Observe that \eqref{alternativeversionofequationhyper} includes an implicit quadratic form summing over the components of $d_{1/2} u$. The equation follows by using Taylor-type expansions. It is clear that we could immeidately estimate \eqref{alternativeversionofequationhyper} directly in the same way as $R_{1}$ in the proof above and obtain uniqueness this way. Moreover, it is easily seen that the gradient flow equation with the non-linearity in \eqref{alternativeversionofequationhyper} and \eqref{gradflowuvhyper} actually are equivalent for sufficiently regular solutions.
	
	\section{Fractional Harmonic Gradient Flow with Values in a General Manifold}
	
	After having treated the case of orientable, closed hypersurfaces in some detail, we perform analogous investigations to resolve the general case and finally also discuss existence and convergence of solutions to the fractional harmonic gradient flow. In our computations and simplifications of the half-harmonic map equation, we follow \cite{mazoschi}, where the half-harmonic map equation was already stated as well as treated and some of its features were already highlighted. We enhance the exposition given there by analysing certain estimates in more detail, leading to the uniqueness result under improved regularity. To prove that energy class solutions actually possess slightly better regularity properties than assumed under some smallness condition on the energy, one proceeds similar to the one for $S^{n-1}$. The major difference lies in the use of Morrey estimates in the case $p = 2$ following \cite{daliopigati}, as we would like to deduce slightly better integrability and then apply the techniques in \cite{riv} for arbitrary integrability of the fractional gradient. This change of method should however not obscure the fact that the regularity gain once more stems from the emergence of an anti-symmetric potential. The discussion of local existence and convergence on the other hand is very similar to \cite{wettstein}, once we fix the right formulation for the $1/2$-harmonic gradient flow.
	
	\subsection{Half-Harmonic Map Equation for general Target Manifolds}
	
	\subsubsection{The Main Equation}
	
	Before we begin our analysis of the fractional harmonic gradient flow, we want to study the $1/2$-harmonic map equation and its features. Throughout this chapter, we assume that $N \subset \R^n$ is a closed and smooth manifold, in particular it is compact and without a boundary. Let us denote by $B_{\delta}(N)$ the $\delta$-neighbourhood of $N$, i.e. the collection of all points in $\R^n$ at a distance $< \delta$ from $N$. This definition obviously makes sense for any $\delta > 0$. By standard theory of smooth manifolds, i.e. the tubular neighbourhood theorem, we know that the closest point projection is a well-defined and smooth map:
	$$\pi: B_{\delta}(N) \to N,$$
	for sufficiently small $\delta > 0$, such that:
	$$\| \pi(x) - x \| = \inf_{y \in N} \| x-y \|, \quad \forall x \in B_{\delta}(N)$$
	One can show that the differential of $\pi$ at any point $x \in N$ is actually the orthogonal projection onto the tangent space of $N$ at $x$. The orthogonality of the projection means that the following holds:
	$$\forall x \in N: \quad d\pi(x)^2 = d\pi(x) = d\pi(x)^{T}$$
	Let $\varphi \in C^{\infty}_{c}(\R^n)$ be a smooth, compactly supported function, such that $\operatorname{supp} \varphi \subset B_{\delta}(N)$ and assume that:
	$$\varphi(x) = 1, \quad \forall x \in B_{\delta/2}(N)$$
	Then we may extend $\pi$ to a map $\tilde{\pi}: \R^n \to \R^n$ by the following formula:
	$$\tilde{\pi}(x) = \varphi(x) \cdot \pi(x), \quad \forall x \in B_{\delta}(N),$$
	as well as $\tilde{\pi}(x) = 0$ for every $x \notin B_{\delta}(N)$. This map is clearly well-defined and smooth. It agrees with $\pi$ on the set $B_{\delta/2}(N)$, but has the advantage of being defined globally, rendering certain definitions and computations easier later on. By abuse of notation, we shall from now on refer to $\tilde{\pi}$ as $\pi$ to keep the notation as simple as possible. It should be noticed at this point that the set of fixed points of $\tilde{\pi}$ consists of $N$ and a set $\tilde{N}$ with positive distance from $N$ due to the definition of the extension of the closest point projection.\\
	
	We recall that a function $u \in {H}^{1/2}(S^1; N)$ is called $1/2$-harmonic, if and only if:
	\begin{equation}
	\label{generalharmmap}
		(-\Delta)^{1/2}u \perp T_{u} N \Leftrightarrow d\pi(u) (-\Delta)^{1/2} u = 0,
	\end{equation}
	see \cite{dalioriv}, \cite{mazoschi} and the references therein. This can be easily verified, as $1/2$-harmonic maps are precisely the critical points in ${H}^{1/2}(S^1;N)$ of the energy:
	$$E(u) := \frac{1}{2} \int_{S^1} | (-\Delta)^{1/4} u |^2 dx$$
	Arguing as for hypersurfaces shows that the orthogonality relation above is another way to phrase this relation. Computing the Euler-Lagrange equation using a variation of the form:
	$$u(t) := \pi(u + t \varphi),$$
	for $\varphi \in C^{\infty}(S^1)$ and $t$ small enough, such that $u + t \varphi \in B_{\delta/2}(N)$, i.e. such that $\pi$ maps $u + t \varphi$ to a point on $N$ a.e. and thus to conclude $u(t) \in H^{1/2}(S^1;N)$. By criticality, the Euler-Langrange equation can be computed using differentiation of $E(u(t))$ with respect to $t$ at $t=0$:
	\begin{align}
		\frac{d}{dt} \left( \frac{1}{2} \int_{S^1} | (-\Delta)^{1/4} u(t) |^2 \right) \Big{|}_{t=0}	&= \int_{S^1} (-\Delta)^{1/4} u \cdot (-\Delta)^{1/4} \left( \frac{d}{dt} u(t) \Big{|}_{t=0} \right) dx \notag \\
																	&= \int_{S^1} (-\Delta)^{1/2} u \cdot d\pi(u) \varphi dx \notag \\
																	&= \int_{S^1} \int_{S^1} d_{1/2} u(x,y) d_{1/2} \left( d\pi(u) \varphi \right)(x,y) \frac{dy dx}{| x-y |},
	\end{align}
	which is precisely how we understand \eqref{generalharmmap}. The product of the vectors is naturally understood as a scalar product in the usual sense.\\
	
	For later use, we shall sometimes write $d\pi^{\perp}(u)$ for the following:
	\begin{equation}
	\label{complproj}
		d\pi^{\perp}(u) := Id - d\pi(u)
	\end{equation}
	One should observe that for every $x \in N$, the differential $d\pi^{\perp}(p)$ is actually another orthogonal projection, meaning:
	$$d\pi^{\perp}(x)^{T} = d\pi^{\perp}(x) = d\pi^{\perp}(x)^{2}$$
	This can be easily deduced from the corresponding identities for $d\pi(x)$.\\
	
	\subsubsection{Rewriting the Half-Harmonic Map Equation}
	
	Our first goal lies in the simplification of the half-harmonic map equation, similar considerations then apply to the $1/2$-harmonic gradient flow which is defined by the equation:
	$$u_{t} + (-\Delta)^{1/2} u \perp T_{u} N, \quad \forall (t,x) \in \R_{+} \times S^1,$$
	see also \cite{wettstein} for the case $N = S^{n-1}$. Therefore, working merely with the $1/2$-harmonic map equation instead of the associated flow is for brevity's sake.\\
	
	In fact, we would like to write the harmonic map equation in a similar form as in the case of $N = S^{n-1}$ or even the case of closed, orientable hypersurfaces, since this form seems appropriate for the proof of a Regularity Lemma similar to Lemma 3.4 in \cite{wettstein} and \cite{riv}. In addition, such considerations are useful when studying uniqueness under improved regularity assumptions, as we have seen both for $N = S^{n-1}$ and $N$ being a closed, orientable hypersurface. Indeed, we rewrote the half-harmonic map equation in both cases in such a way that enabled us to estimate certain summands more easily and reveal a quadratic structure in the non-local PDE. The computations we make are precisely the ones found \cite{mazoschi}, we merely rewrite certain terms in a manner more appropriate for our purposes and provide more general estimates than the ones found in \cite{mazoschi}.\\
	
	First of, let $\varphi \in C^{\infty}(S^1)$. Then we know by using the fractional version of Leibniz' rule as seen for example in \eqref{fractionalleibnizruleintrod} with $s = 1/2$:
	\begin{align}
	\label{rewriting1}
		\int_{S^1} d_{1/2} u \cdot d_{1/2}\varphi dx	&= \int_{S^1} d_{1/2} u \cdot d_{1/2} \left( d\pi(u) \varphi \right) dx + \int_{S^1} d_{1/2} u \cdot d_{1/2} \left( d\pi^{\perp}(u) \varphi \right) dx \notag \\
											&= \int_{S^1} d_{1/2} u \cdot d_{1/2} \left( d\pi^{\perp}(u) \varphi \right) dx \notag \\
											&= \int_{S^1} \int_{S^1} d_{1/2} u(x,y) d_{1/2} ( d\pi^{\perp}(u) )(x,y) \frac{dy}{| x-y |} \varphi(x) dx \notag \\
											&+ \int_{S^1} \int_{S^1} d_{1/2} u(x,y) d\pi^{\perp}(u(y)) d_{1/2}\varphi(x,y) \frac{dy dx}{| x-y |} \notag \\
	\end{align}
	Observe that we used \eqref{generalharmmap} in the second step for $u$ half-harmonic. We shall refer to the second summand as $R_1(\varphi)$:
	$$R_1(\varphi) := \int_{S^1} \int_{S^1} d_{1/2} u(x,y) d\pi^{\perp}(u(y)) d_{1/2}\varphi(x,y) \frac{dy dx}{| x-y |}$$
	The goal is to treat $R_1$ as a perturbation of a "main order term", which we consider to be the remaining one in \eqref{rewriting1}. To achieve this, we need the following Lemma that was already stated in \cite{mazoschi}:
	
	\begin{lem}
	\label{taylorpi}
		Let $x, y \in S^1$ and $u \in H^{1/2}(S^1;N)$. Then the following holds for all $x,y \in S^1$ such that $u(x), u(y) \in N$:
		$$u(x) - u(y) = d\pi(u(y))\left(u(x) - u(y) \right) + \int_{0}^1 \int_{0}^1 t d \left( d\pi \right)\left( (1-ts) u(y) + ts u(x) \right) (u(x) - u(y))^2 ds dt$$
		More precisely, for $i = 1, \ldots, n$, we have for the $i$-th component of $u = (u_{1}, \ldots, u_{n})$:
		\begin{align*}
			u_{i}(x) - u_{i}(y) 	&= d\pi(u(y))_{ij} (u_{j}(x) - u_{j}(y)) \\
							&+ \int_{0}^1 \int_{0}^1 t \partial_{kj} \pi_{i} \left( (1-ts) u(y) + tsu(x) \right) (u_{j}(x) - u_{j}(y))(u_{k}(x) - u_{k}(y)) ds dt,
		\end{align*}
		adopting Einstein's summation convention.
	\end{lem}
	
	Similar identities and/or estimates already appeared in \cite{daliopigati}, \cite{mazoschi}. We shall define the following quantity based on Lemma \ref{taylorpi}:
	\begin{equation}
	\label{curvature}
		A_{u}(dv, dw)(x,y) :=  \int_{0}^1 \int_{0}^1 t d \left( d\pi \right)\left( (1-ts) u(y) + ts u(x) \right) (v(x) - v(y))(w(x) - w(y)) ds dt
	\end{equation}
	This simplifies the result above considerably. In fact, we could restate Lemma \ref{taylorpi} as:
	$$u_{i}(x) - u_{i}(y) = d\pi(u(y))_{ij} (u_{j}(x) - u_{j}(y)) + A^{i}_{u}(du, du)(x,y),$$
	We emphasise that this is one of the main points why we would like to consider $\pi$ as a map defined on all of $\R^n$. This way, we may insert any value into the function, meaning that even if $(1-ts)u(y) + ts u(x)$ is not necessarily in $N$ or even close enough for the closest point projection to be well-defined, the formula above still produces a reasonable result we may work with. The proof is an immediate and standard application of Taylor's formula and/or repeated applications of the fundamental theorem of calculus and therefore omitted. We refer to \cite{mazoschi} for details.\\
	
	Using the orthogonality of the projections $d\pi$ and $d\pi^{\perp}$, we may deduce (again using Einstein's summation convention):
	\begin{align}
	\label{rest01}
		R_{1}(\varphi)	&= \int_{S^1} \int_{S^1} \frac{u_{i}(x) - u_{i}(y)}{| x-y |^{1/2}} d\pi^{\perp}(u(y))_{ij} d_{1/2}\varphi_{j}(x,y) \frac{dy dx}{| x-y |} \notag \\
					&= \int_{S^1} \int_{S^1} \frac{d\pi(u(y))_{ik} (u_{k}(x) - u_{k}(y))}{| x-y |^{1/2}} d\pi^{\perp}(u(y))_{ij} d_{1/2}\varphi_{j}(x,y) \frac{dy dx}{| x-y |} \notag \\
					&+ \int_{S^1} \int_{S^1} \frac{A^{i}_{u}(du, du)(x,y)}{| x-y |^{1/2}} d\pi^{\perp}(u(y))_{ij} d_{1/2}\varphi_{j}(x,y) \frac{dy dx}{| x-y |} \notag \\
					&= \int_{S^1} \int_{S^1} \frac{A^{i}_{u}(du, du)(x,y)}{| x-y |^{1/2}} d\pi^{\perp}(u(y))_{ij} d_{1/2}\varphi_{j}(x,y) \frac{dy dx}{| x-y |},
	\end{align}
	since:
	$$d\pi(u(y))_{ik} d\pi^{\perp}(u(y))_{ij} = 0, \quad \forall k,j,$$
	due to orthogonality of the projection and $d\pi^{\perp} = Id - d\pi$. We observe that the formula above for $R_{1}$ could also be stated as:
	$$R_1 = \div_{1/2} \left( \frac{A^{i}_{u}(du, du)(x,y)}{| x-y |^{1/2}} d\pi^{\perp}(u(y))_{ij} \right)$$
	This combined with the computations in \eqref{rewriting1} show us that the fractional harmonic map equation can actually be rephrased as:
	\begin{equation}
		(-\Delta)^{1/2} u = d_{1/2} u \cdot d_{1/2} \left( d\pi^{\perp}(u) \right) + \div_{1/2} \left( \frac{A^{i}_{u}(du, du)(x,y)}{| x-y |^{1/2}} d\pi^{\perp}(u(y))_{ij} \right)
	\end{equation}
	Comparing this with the previously investigated case of orientable, closed hypersurfaces and keeping \eqref{factorlambda} in mind, we notice the immediate similarities between the two equations. Indeed, this will be a crucial point in reducing our computations to establish uniqueness under improved regularity and obtaining the result basically for free from what we have already done.\\
	
	We may prove the following estimates:
	
	\begin{lem}
	\label{refr1}
		Let $\varphi \in \dot{F}^{1/2}_{p',2}(S^1)$ for $1/p + 1/p' = 1$ and $p \geq 2$ finite. Then, for all $v \in \dot{F}^{1/2}_{2,2}(S^1), w \in \dot{F}^{1/2}_{p,2}(S^1)$, we have:
		\begin{equation}
		\label{estr1higher}
			| R_{1}^{v,w}(\varphi) | := \left| \int_{S^1} \int_{S^1} \frac{A^{i}_{u}(dv, dw)(x,y)}{| x-y |^{1/2}} d\pi^{\perp}(u(y))_{ij} d_{1/2} \varphi_{j}(x,y) \frac{dy dx}{| x-y |} \right| \lesssim \| \varphi \|_{\dot{F}^{1/2}_{p',2}} \| v \|_{\dot{F}^{1/2}_{2,2}} \| w \|_{\dot{F}^{1/2}_{p,2}}
		\end{equation}
		In addition, for $v,w \in \dot{F}^{1/2}_{4,2}(S^1)$, we have:
		\begin{align*}
			\int_{S^1} \int_{S^1}	&\frac{A^{i}_{u}(dv, dw)(x,y)}{| x-y |^{1/2}} d\pi^{\perp}(u(y))_{ij} d_{1/2} \varphi_{j}(x,y) \frac{dy dx}{| x-y |} \notag \\
							&= \int_{S^1} \int_{S^1} {A^{i}_{u}(dv, dw)(x,y)} d\pi^{\perp}(u(y))_{ij} - {A^{i}_{u}(dv, dw)(y,x)} d\pi^{\perp}(u(x))_{ij} \frac{dy}{| x-y |^2} \varphi_{j}(x) dx,
		\end{align*}
		and as a result $R_{1}^{v,w} \in L^{2}(S^1)$ with:
		\begin{equation}
		\label{estr1reg}
			\| R_1 \|_{L^2} \lesssim \| v \|_{\dot{F}^{1/2}_{4,2}} \| w \|_{\dot{F}^{1/2}_{4,2}}
		\end{equation}
	\end{lem}
	
	\begin{proof}
		The first estimate follows by letting $0 < s < 1$ and observing:
		\begin{align}
		\label{holderapplication}
			\Big{|} \int_{S^1} \int_{S^1} 	&\frac{A^{i}_{u}(dv, dw)(x,y)}{| x-y |^{1/2}} d\pi^{\perp}(u(y))_{ij} d_{1/2} \varphi(x,y) \frac{dy dx}{| x-y |} \Big{|}	\notag \\
									&\lesssim \left|  \int_{S^1} \int_{S^1} | d_{s} v(x,y) | | d_{1/2 - s} w(x,y) | | d_{1/2} \varphi(x,y) | \frac{dy dx}{| x-y |}\right| \notag \\
									&\lesssim \| v \|_{\dot{W}^{s, (1/s, 4)}} \| w \|_{\dot{W}^{1/2 - s, (p/(1-sp), 4)}} \| \varphi \|_{\dot{W}^{1/2, (p', 2)}},
		\end{align}
		where we used the triangle inequality and H\"older's inequality. If we can choose $s$ in a way, such that:
		$$\frac{p}{1-sp} > \frac{4}{3 - 4s}, \frac{1}{s} > \frac{4}{1 + 4s},$$
		we may use the identification mentioned in the preliminary section. One notices that the latter inequality is trivially true, while the first one reduces to:
		$$p \geq 2 > \frac{4}{3},$$
		which again holds trivially. So the choice of $s \in (0,1/2)$ is immaterial here (we just need to ensure that we do not divide by $0$ or have negative H\"older exponents). Indeed these conditions lead to $0 < s < 1/p$, which clearly allows for a choice of $s$. Consequently, we may identify:
		$$\dot{W}^{s, (1/s, 4)} = \dot{F}^{s}_{1/s, 4}, \quad \dot{W}^{1/2 - s, (p/(1-sp), 4)} = \dot{F}^{1/2 - s}_{p/(1-sp), 4}, \quad \dot{W}^{1/2, (p', 2)} = \dot{F}^{1/2}_{p', 2},$$
		and the norms are equivalent, as mentioned in the preliminary section, see \cite{schiwang}. By using Sobolev embeddings for Triebel-Lizorkin spaces, this shows:
		$$\| v \|_{\dot{F}^{s}_{1/s, 4}} \lesssim \| v \|_{\dot{F}^{s}_{1/s, 2}} \lesssim \| v \|_{\dot{F}^{1/2}_{2, 2}},$$
		as well as:
		$$\| w \|_{\dot{F}^{1/2 - s}_{p/(1-sp), 4}} \lesssim \| w \|_{\dot{F}^{1/2 - s}_{p/(1-sp), 2}} \lesssim \| w \|_{\dot{F}^{1/2}_{p, 2}}.$$
		Combining these inequalities with the estimate in \eqref{holderapplication}, we find the desired result:
		$$\Big{|} \int_{S^1} \int_{S^1} \frac{A^{i}_{u}(dv, dw)(x,y)}{| x-y |^{1/2}} d\pi^{\perp}(u(y))_{ij} d_{1/2} \varphi(x,y) \frac{dy dx}{| x-y |} \Big{|} \lesssim \| \varphi \|_{\dot{F}^{1/2}_{p',2}} \| v \|_{\dot{F}^{1/2}_{2,2}} \| w \|_{\dot{F}^{1/2}_{p,2}}$$
		For the second identity mentioned above, one observes that the equality mentioned holds by a simple change of role between $x$ and $y$. The estimate is then obtained by noticing:
		\begin{align}
			\Big{|} \int_{S^1} \int_{S^1} 	&{A^{i}_{u}(dv, dw)(x,y)} d\pi^{\perp}(u(y))_{ij} - {A^{i}_{u}(dv, dw)(y,x)} d\pi^{\perp}(u(x))_{ij} \frac{dy}{| x-y |^2} \varphi_{j}(x) dx \Big{|} \notag \\
									&\lesssim \int_{S^1} \int_{S^1} | d_{1/2} v(x,y) | | d_{1/2} w(x,y) | \frac{dy}{| x-y |} | \varphi_{j}(x) | dx \notag \\
									&\lesssim \int_{S^1} | d_{1/2} v |(x) | d_{1/2} w |(x) \frac{dy}{| x-y |} | \varphi_{j}(x) | dx \notag \\
									&\lesssim \| | d_{1/2} v |(x) \|_{L^{4}} \| | d_{1/2} w |(x) \|_{L^{4}} \| \varphi \|_{L^{2}} \notag \\
									&\lesssim \| v \|_{\dot{F}^{1/2}_{4,2}} \| w \|_{\dot{F}^{1/2}_{4,2}} \| \varphi \|_{L^2},
		\end{align}
		by applying H\"older's inequality twice as well as boundedness of $d\pi$ and consequently for $d \pi^{\perp}$. Observe that we used once more the equivalent characterisation of the Triebel-Lizorkin norm in the case $q = 2, s = 1/2$.
	\end{proof}
	
	This representation of the fractional harmonic map equation will already suffice for proving uniqueness under improved regularity assumptions.\\
	
	We shall continue as in \cite{mazoschi} and further explore simplifications of the leading term, i.e. the term:
	$$\int_{S^1} \int_{S^1} d_{1/2} u_{i}(x,y) d_{1/2} ( d\pi^{\perp}(u) )_{ij}(x,y) \frac{dy}{| x-y |} \varphi_{j}(x) dx = \int_{S^1} d_{1/2} u_{i} \cdot d_{1/2} ( d\pi^{\perp}(u) )_{ij}(x)  \varphi_{j}(x) dx$$
	Using Lemma \ref{taylorpi}, we may replace $d_{1/2} u(x,y)$ by the following expression:
	\begin{align*}
		\int_{S^1} &d_{1/2} u_{i}(x,y) d_{1/2} ( d\pi^{\perp}(u) )_{ij}(x,y) \frac{dy}{| x-y |}	\notag \\
																	&= \int_{S^1} d_{1/2} \left( d\pi^{\perp}(u) \right)_{ij} (x,y) d\pi_{ik}(u(y)) d_{1/2} u_{k}(x,y) \frac{dy}{| x-y |} \notag \\
																	&+ \int_{S^1} d_{1/2} \left( d\pi^{\perp}(u) \right)_{ij} (x,y) A^{i}_{u}(du, du)(x,y) \frac{dy}{| x-y |^{3/2}} 
	\end{align*}
	The second summand is defined to be:
	\begin{equation}
	\label{rest02}
		R_2(x) := \int_{S^1} d_{1/2} \left( \pi^{\perp}(u) \right)_{ij}(x,y) A^{i}_{u}(du, du)(x,y) \frac{dy}{| x-y |^{3/2}}
	\end{equation}
	The first summand may be rewritten one last time to obtain:
	\begin{equation}
		\int_{S^1} d_{1/2} \left( \pi^{\perp}(u) \right)_{ij}(x,y) d\pi_{ik}(u(y)) d_{1/2} u_{k}(x,y) \frac{dy}{| x-y |} = \Omega_{jk} \cdot d_{1/2}u_{k} + R_3,
	\end{equation}
	where:
	\begin{equation}
	\label{defomega}
		\Omega_{jk}(x,y) := {d\pi(u(y))_{ik} d_{1/2}\left( d\pi^{\perp}(u) \right)_{ij}(x,y) - d\pi(u(y))_{ij} d_{1/2}\left( d\pi^{\perp}(u) \right)_{ik}(x,y)}, \quad \forall j, k,
	\end{equation}
	as well as:
	\begin{align}
	\label{rest03}
		R_{3}(x) 	&:= \int_{S^1} d\pi(u(y))_{ij} d_{1/2}\left( d\pi^{\perp}(u) \right)_{ik}(x,y) d_{1/2}u_{k}(x,y) \frac{dy}{| x-y |} \notag \\
				&= \int_{S^1} d\pi(u(y))_{ij} d\pi^{\perp}(u(x))_{ik} d_{1/2}u_{k}(x,y) \frac{dy}{| x-y |^{3/2}} \notag \\
				&= -\int_{S^1} d_{1/2} \left( d\pi(u) \right)_{ij}(x,y) d\pi^{\perp}(u(x))_{ik} d_{1/2}u_{k}(x,y) \frac{dy}{| x-y |},
	\end{align}
	simply because of orthogonality of the projections. We highlight that $\Omega = -\Omega^{T}$, i.e. we have found an anti-symmetric potential similar to the case $N = S^{n-1}$. We will sometimes write $\Omega_u$ to emphasise the dependence on $u$. This also suggests that an increase in integrability should be obtainable in the critical case of the non-local PDE, i.e. for $u \in H^{1/2}(S^1)$.\\
	
	We close this preliminary examination of the half-harmonic map equation by summarising these computations in the following equivalent equation for fractional harmonic maps:
	\begin{equation}
	\label{generalharmmaprewritten}
		(-\Delta)^{1/2} u = \Omega \cdot d_{1/2} u + R_{1} + R_{2} + R_{3},
	\end{equation}
	where we use the definitions \eqref{rest01}, \eqref{rest02}, \eqref{rest03} as well as \eqref{defomega}. Let us remark that the following holds:
	
	\begin{lem}
	\label{refr2}
		Let $v \in \dot{F}^{1/2}_{p,2}(S^1), w \in \dot{F}^{1/2}_{2,2}(S^1)$ with $p > 2$. Then:
		$$\| R_{2}^{v,w} \|_{L^{\frac{2p}{2+p}}} \lesssim \| u \|_{\dot{F}^{1/2}_{2,2}} \| v \|_{\dot{F}^{1/2}_{p,2}} \| w \|_{\dot{F}^{1/2}_{2,2}},$$
		and if $v, w \in \dot{F}^{1/2}_{4,2}$
		$$\| R_{2}^{v,w} \|_{L^{2}} \lesssim \| u \|_{\dot{F}^{1/2}_{2,2}} \| v \|_{\dot{F}^{1/2}_{4,2}} \| w \|_{\dot{F}^{1/2}_{4,2}},$$
		where:
		$$R_{2}^{v,w} := \int_{S^1} d_{1/2} \left( \pi^{\perp}(u) \right)_{ij}(x,y) A^{i}_{u}(dv, dw)(x,y) \frac{dy}{| x-y |^{3/2}}$$
		Similar estimates can be obtained for $R_3$ by using a variant of Lemma \ref{taylorpi}.
	\end{lem}
	
	The proof is similar to the one in Lemma \ref{refr1}. One notices that by Sobolev embeddings for Triebel-Lizorkin spaces $\dot{F}^{1/2}_{p',2}(S^1) \subset L^{\frac{2p}{p-2}}(S^1)$, we may deduce that:
	$$R_{2}^{v,w}, R_{3}^{v,w} \in \dot{F}^{-1/2}_{p,2}(S^1),$$
	with estimates analogous to the ones in Lemma \ref{refr2}.
	
	\begin{proof}
		We see:
		$$| R_{2}^{v,w}(x) | \lesssim \int_{S^1} | d_{s} u(x,y) | | d_{s} w(x,y) | | d_{1-2s} v(x,y) | \frac{dy}{| x-y |},$$
		and therefore, by using H\"older's inequality:
		\begin{align}
			\| R_{2}^{v,w} \|_{L^{\frac{2p}{p+2}}} 	&\lesssim \| u \|_{\dot{W}^{s, (1/s,3)}} \| w \|_{\dot{W}^{s, (1/s,3)}} \| v \|_{\dot{W}^{1-2s, (2p/(p+2-4sp),3)}} \notag \\
										&\lesssim \| u \|_{\dot{F}^{1/2}_{2,2}} \| w \|_{\dot{F}^{1/2}_{2,2}} \| v \|_{\dot{F}^{1/2}_{p,2}},
		\end{align}
		where we used Sobolev embeddings for Triebel-Lizorkin spaces in the last step. We emphasise that changing between the spaces $\dot{W}^{s,(p,q)}$ and $\dot{F}^{s}_{p,q}$ is possible by \cite{schiwang}, see Theorem \ref{schiwangthm1.4} and the comment afterwards in section 2, due to:
		$$\frac{1}{s} > \frac{3}{1 + 3s}, \quad \frac{2p}{p+2-4ps} > \frac{3}{4-6s},$$
		of which the first inequality is trivially true and the latter reduces to:
		$$8p - 12sp > 3p + 6 - 12ps \Rightarrow p > \frac{6}{5},$$
		which is trivially true for all $s$ because $p \geq 2$. So one merely has to take care that the H\"older exponents remain in $(1, +\infty)$, which is easily ensured as in Lemma \ref{refr1}. The second estimate is obtained along the same lines, merely changing Triebel-Lizorkin spaces. The estimates for $R_3$ follow from similar considerations by using Lemma \ref{taylorpi}, but this time by expanding around $u(x)$ instead of $u(y)$ which only requires minor modifications. Thus we are done.
	\end{proof}
	
	As in \cite{wettstein}, we may also find the corresponding fractional harmonic gradient flow to be therefore:
	\begin{align}
	\label{deductionoffracgradflow}
		u_{t} + (-\Delta)^{1/2} u 	&= d_{1/2} u \cdot d_{1/2} \left( d\pi^{\perp}(u) \right) + \div_{1/2} \left( \frac{A^{i}_{u}(du, du)(x,y)}{| x-y |^{1/2}} d\pi^{\perp}(u(y))_{ij} \right) \\
							&= \Omega_{u} \cdot d_{1/2} u + R_1 + R_2 + R_3,
	\end{align}
	where the latter equation uses the expressions for $\Omega_u, R_1, R_2, R_3$ introduced above.
	
	\subsubsection{Other useful Formulations for the Fractional Harmonic Gradient Flow}
	
	In the proof of existence of local solutions to the fractional harmonic gradient flow, we shall make use of a slightly different formulation for the $1/2$-harmonic gradient flow. Therefore, we list already here several equivalent (at least for sufficiently regular $u$) formulations of the main equation \eqref{deductionoffracgradflow}:\\
	
	The most basic formulation of the gradient flow we are studying is:
	$$u_t + (-\Delta)^{1/2} u \perp T_u N,$$
	which characterises solutions of the gradient flow with values in $N$ a.e.. This can be rewritten as:
	\begin{equation}
		u_t + (-\Delta)^{1/2} u = (-\Delta)^{1/2} \pi(u) - d\pi(u) (-\Delta)^{1/2} u,
	\end{equation}
	where we use for $u \in N$ almost everywhere:
	\begin{align}
		(-\Delta)^{1/2} \pi(u) - d\pi(u) (-\Delta)^{1/2} u	&= (-\Delta)^{1/2} u - d\pi(u) (-\Delta)^{1/2} u \notag \\
											&= \left( Id - d\pi(u) \right) (-\Delta)^{1/2} u \notag \\
											&= d\pi^{\perp}(u) (-\Delta)^{1/2} u \perp T_{u} N,
	\end{align}
	which is in fact the same as $d\pi^{\perp}(u) (u_t + (\Delta)^{1/2} u )$, since $u_t \in T_u N$ a.e.. A key advantage of this formulation is that we may encode the condition $u \in N$ a.e. directly inside the equation, see the proof of local existence of solutions below. A major drawback on the other hand is that this formulation obscures the compensation phenomena at hand. Lastly, there also exists a formulation analogous to \eqref{alternativeversionofequationhyper} by using arguments analogous to the ones needed to prove Lemma \ref{taylorpi}. This leads to:
	\begin{equation}
	\label{umformulierungenmitquadratischerform}
		u_{t} + (-\Delta)^{1/2} u = \sum_{k,l = 1}^{n} P.V. \int_{S^1} P^{kl}(u(x),u(y)) d_{1/2} u_k (x,y) d_{1/2} u_l (x,y) \frac{dy}{| x-y |},
	\end{equation}
	where $P^{kl} = (P^{kl}_1, \ldots, P^{kl}_n)$ with:
	\begin{align}
		&P_{j}^{kl}(u(x),u(y)) (u_{k}(x) - u_{k}(y)) (u_{l}(x) - u_{l}(y)) \notag \\ 
		&=  \sum_{k=1}^{n} \sum_{l=1}^{n} \int_{0}^{1} \int_{0}^{1} (t-1) \partial_{kl} \pi_{j} ((s-st) u(y) + (1 + st - s) u(x)) (u_{k}(x) - u_{k}(y)) (u_{l}(x) - u_{l}(y)) ds dt.
	\end{align}
	One observes the immediate similarity with \eqref{alternativeversionofequationhyper} as well as the quadratic structure of the RHS of this formulation. This formulation is again useful in proving uniqueness and regularity, but the connection to the formulation in \cite{wettstein} for $N = S^{n-1}$ is less apparent than with \eqref{deductionoffracgradflow}.

	\subsection{Uniqueness of Solutions to the $1/2$-Harmonic Gradient Flow}
	
	We discuss now the uniqueness of solutions to the $1/2$-harmonic gradient flow. Our approach is similar to the one in \cite{wettstein}. Thus, we first discuss uniqueness in a class of functions which have more regularity than strictly required to make sense of solutions and follow the approach in \cite{struwe} to establish uniqueness for such "strong solutions". Then, we expand this result to arbitrary energy class solution, i.e. the class of functions with minimal regularity for the fractional gradient flow to make sense, by using techniques similar to \cite{riv} to squeeze some more integrability out of solutions with small energy. In fact, we will rely on the techniques in \cite{daliopigati} which will turn out to be important to establish a slight gain in regularity by means of compensation phenomena tied to the emergence of an antisymmetric potential (notice that the appearance of such a potential from a slightly different point of view is already hinted at in Section 4.1 and in \cite{mazoschi}).
	
	\subsubsection{Uniqueness under Improved Regularity Assumptions}
	
	We are now able to turn to the study of the gradient flow associated with the fractional harmonic map with values in $N \subset \R^n$ being a general closed manifold. As previously in the case of a closed orientable hypersurface, let us assume that $u, v$ are two solutions to the fractional harmonic gradient flow taking a.e. values in a general closed $N \subset \R^n$. Clearly, this implies boundedness of $u,v$ due to the compactness of $N$. As before in Section 3, we assume that the following regularity conditions hold:
	\begin{equation}
	\label{regassgeneral}
		u,v \in L^{\infty}(\R_{+}; H^{1/2}(S^1)); \quad u_t, v_t \in L^{2}(\R_{+}; L^{2}(S^{1})); \quad u,v \in L^{2}_{loc}(\R_{+}; H^{1}(S^{1})),
	\end{equation}
	In addition, they satisfy the gradient flow associated with the $1/2$-harmonic map as described below:
	\begin{equation}
	\label{gradflowuvgeneral}
		w_t + (-\Delta)^{1/2} w = d_{1/2} w \cdot d_{1/2} \left( d\pi^{\perp}(w) \right) + \div_{1/2} \left( \frac{A^{i}_{w}(dw, dw)(x,y)}{| x-y |^{1/2}} d\pi^{\perp}(w(y))_{ij} \right),
	\end{equation}
	for both $w = u$ and $w = v$, together with the boundary condition $u(0, \cdot) = v(0, \cdot) = u_{0} \in H^{1/2}(S^1)$. We mention that, as we have seen in the previous subsection, the right hand side of \eqref{gradflowuvgeneral} is precisely the non-linearity associated with the fractional harmonic map equation for functions taking values in $N$. The equation \eqref{gradflowuvgeneral} could be derived along the same lines, cf. \cite{wettstein}.\\
	\noindent
	Let us present the main uniqueness statement in analogy to Theorem \ref{uniquehyper}:
	
	\begin{thm}
	\label{uniquegeneral}
		If $u,v$ both solve \eqref{gradflowuvgeneral} with the same initial datum $u_0 \in H^{1/2}(S^1;N)$ and we assume that:
		$$\| (-\Delta)^{1/4} u(t) \|_{L^{2}(S^1)}, \| (-\Delta)^{1/4} v(t) \|_{L^{2}(S^1)} \leq \| (-\Delta)^{1/4} u_{0} \|_{L^{2}(S^1)}, \quad \forall t \in \R_{+},$$
		then we have:
		$$u = v$$
	\end{thm}
	
	The proof is actually going to proceed analogous to the case of closed orientable hypersurfaces. We will provide some of the details below and the proof naturally extends to the case of solutions $u,v$ defined only on a subinterval $[0,T] \subset \R_{+}$.
	
	\begin{proof}
		First, we observe that since $d\pi^{\perp} = Id - d\pi$, we know:
		$$d_{1/2} \left( d\pi^{\perp}(u) \right)(x,y) = -d_{1/2} \left( d\pi(u) \right)(x,y),$$
		and we therefore would like to estimate the following:
		\begin{equation}
			d\pi(u(x)) - d\pi(u(y)) = \int_{0}^{1} d\left( d\pi \right)\left( (1-s) u(y) + s u(x) \right) ds \cdot (u(x) - u(y)),
		\end{equation}
		using Taylor expansion and understanding the differential $d (d\pi)$ as previously in Lemma \ref{taylorpi}. If we define:
		$$B_{u}(x,y) := \int_{0}^{1} d\left( d\pi \right)\left( (1-s) u(y) + s u(x) \right) ds,$$
		which is clearly bounded thanks to the smoothness of $\pi$ and its definition as an extension, we may rewrite:
		$$d_{1/2} u \cdot d_{1/2} \left( d\pi^{\perp}(u) \right) = -d_{1/2} u \cdot B_{u}(x,y) d_{1/2} u,$$
		which renders the fractional harmonic flow equation virtually the same as previously in the case of $N$ being an orientable closed hypersurface:
		$$u_t + (-\Delta)^{1/2} u = -d_{1/2} u \cdot B_{u}(x,y) d_{1/2} u + \div_{1/2} \left( \frac{A^{i}_{u}(du, du)(x,y)}{| x-y |^{1/2}} d\pi^{\perp}(u(y))_{ij} \right),$$
		and completely analogous for $v$.\\
		
		We may now use a decomposition as in \eqref{decompr1r2forhyper}. The first remainder involving $B_{u}(x,y)$ can be estimated analogous to \eqref{decompr1hyper} by obvious modifications of the estimates provided there. The second remainder, i.e. the fractional divergence, has a similar form to \eqref{decompr2equation} and may be decomposed as in \eqref{decompr2hyper}:
		\begin{align}
			&\frac{A^{i}_{u}(du, du)(x,y)}{| x-y |^{1/2}} d\pi^{\perp}(u(y))_{ij} d_{1/2} w - \frac{A^{i}_{v}(dv, dv)(x,y)}{| x-y |^{1/2}} d\pi^{\perp}(v(y))_{ij} d_{1/2} w \notag \\
				&= \frac{A^{i}_{u}(du, du)(x,y)}{| x-y |^{1/2}} d\pi^{\perp}(u(y))_{ij} d_{1/2} w - \frac{A^{i}_{v}(du, du)(x,y)}{| x-y |^{1/2}} d\pi^{\perp}(u(y))_{ij} d_{1/2} w \notag \\
				&+ \frac{A^{i}_{v}(du, du)(x,y)}{| x-y |^{1/2}} d\pi^{\perp}(u(y))_{ij} d_{1/2} w - \frac{A^{i}_{v}(du, du)(x,y)}{| x-y |^{1/2}} d\pi^{\perp}(v(y))_{ij} d_{1/2} w \notag \\
				&+ \frac{A^{i}_{v}(du, du)(x,y)}{| x-y |^{1/2}} d\pi^{\perp}(v(y))_{ij} d_{1/2} w - \frac{A^{i}_{v}(dv, du)(x,y)}{| x-y |^{1/2}} d\pi^{\perp}(v(y))_{ij} d_{1/2} w \notag \\
				&+ \frac{A^{i}_{v}(dv, du)(x,y)}{| x-y |^{1/2}} d\pi^{\perp}(v(y))_{ij} d_{1/2} w - \frac{A^{i}_{v}(dv, dv)(x,y)}{| x-y |^{1/2}} d\pi^{\perp}(v(y))_{ij} d_{1/2} w
		\end{align}
		One may now estimate these terms as in the case $N$ a hypersurface summand by summand. For example, in the first summand we may estimate the difference between $A_{u}$ and $A_{v}$ by $| w(x) | + | w(y) |$ and estimate one of the $d_{0} u$ by its $L^{\infty}$-bound to arrive at an estimate of the form:
		\begin{align*}
			\Big{|} \int_{S^1} \int_{S^1} \frac{A^{i}_{u}(du, du)(x,y)}{| x-y |^{1/2}} d\pi^{\perp}(u(y))_{ij} d_{1/2} w(x,y) &- \frac{A^{i}_{v}(du, du)(x,y)}{| x-y |^{1/2}} d\pi^{\perp}(u(y))_{ij} d_{1/2} w(x,y) dx dy\Big{|} \\
				&\lesssim \int_{S^{1}} | d_{1/2}u |(x) | d_{1/2} w|(x) | w(x) | dx
		\end{align*}
		Observe that we have to exchange labels $x$, $y$ at some point. The other terms can be estimated in a similar manner.\\
		
		The estimates are therefore obtained completely analogous to the case $N$ a hypersurface, applying Young's and H\"older's inequality, leading to an estimate for $w = u - v$ of the following form:
		\begin{align}
			\frac{1}{2} \| w(T) \|_{L^{2}(S^{1})} 	&+ \frac{1}{2} \int_{0}^{T} \| (-\Delta)^{1/4} w(t) \|_{L^{2}(S^{1})} dt \notag \\
										&\leq \tilde{C} \left( \int_{0}^{T} \int_{S^{1}} | w |^{4} dx dt \right)^{1/2} \cdot \left(  \int_{0}^{T} \int_{S^{1}} \left( | d_{1/2} u | + | d_{1/2} v | \right)^{4} dx dt \right)^{1/2}
		\end{align}
		which can then be treated as in $N = S^{n-1}$ in order to conclude uniqueness by an iteration argument, we refer to \cite{wettstein} for the details. The estimates invoked are independent of $S^{n-1}$ and rely on general properties of $u$ provided by \eqref{regassgeneral}, therefore generalising to our current situation.
	\end{proof}
	
	Naturally, as in the case of $N$ being a hypersurface, one may also use the formulation in \eqref{umformulierungenmitquadratischerform} to deduce uniqueness, see also Section 3 for some more details. The proof proceeds analogously and is omitted.

	\subsubsection{Small Initial Energy: Regularity Lemma in the Case $p > 2$}
	
	To deduce uniqueness of fractional gradient flow in energy class under some additional assumption, we have to establish some higher regularity for energy class solutions of the fractional gradient flow. In the case of the $n-1$-sphere, we managed to achieve slightly better regularity properties of the solution by using Wente-type estimates from \cite{mazoschi} and the invertibility of certain operators as in \cite{riv}. In the general case, we will also prove an existence and uniqueness result for a modified operator under higher integrability, which will then be useful to establish the higher regularity needed. In a second step, we shall show that higher integrability actually holds for solutions in ${H}^{1/2}(S^1)$ by means of compensation-compactness as in \cite{daliopigati}, completing the proof.\\
	
	A key Lemma is the following:
	
	\begin{lem}
	\label{casepg2}
		Let $u \in {H}^{1/2}(S^1)$ and $f \in L^2(S^1)$ and assume that $u$ solves:
		\begin{equation}
		\label{reglemmapg2caseequation}
			(-\Delta)^{1/2} u = d_{1/2} u \cdot d_{1/2} \left( d\pi^{\perp}(u) \right) + \div_{1/2} \left( \frac{A^{i}_{u}(du, du)(x,y)}{| x-y |^{1/2}} d\pi^{\perp}(u(y))_{ij} \right) + f,
		\end{equation}
		Then, if there exists a $p > 2$, such that (up to adding a constant):
		$$u \in \dot{F}^{1/2}_{p,2},$$
		then we immediately conclude:
		$$u \in \dot{F}^{1/2}_{q,2}(S^1), \quad \forall q \geq 2.$$
	\end{lem}
	
	The proof relies on the techniques in \cite{riv} and using the remainders and their associated estimates established earlier, see Lemma \ref{refr1} and \ref{refr2}. One should observe that \eqref{gradflowuvgeneral} is naturally of the form \eqref{reglemmapg2caseequation} for almost every fixed time $t \in \R_{+}$. This is apparent by chosing $f = -\partial_t u$.
	
	\begin{proof}
	First, we write:
	$$(-\Delta)^{1/2} u = \Omega \cdot d_{1/2} u + R_1 + R_2 + R_3 + f,$$
	where $\Omega, R_1, R_2, R_3$ are all as previously introduced in this chapter, see \eqref{defomega} as well as \eqref{rest01}, \eqref{rest02}, \eqref{rest03}. We may approximate $\Omega$ by a smooth $\tilde{\Omega}$ vanishing in a neighbourhood of the diagonal and lying in $L^{2}_{od}(S^1 \times S^1)$ and $u$ by a smooth function $\tilde{u}$ in $H^{1/2}(S^1)$, the norms of the differences being arbitrarily small in the respective spaces:
	$$\| \Omega - \tilde{\Omega} \|_{L^{2}_{od}}, \| u - \tilde{u} \|_{H^{1/2}} < \varepsilon,$$
	for $\varepsilon > 0$ to be determined later. Then, we may rewrite this equation as:
	\begin{align}
	\label{maineqgeninv}
		(-\Delta)^{1/2} u + (\tilde{\Omega} - \Omega) \cdot d_{1/2} u 	&+ R_{1}^{\tilde{u} - u, u} + R_{2}^{\tilde{u} - u, u} + R_{2}^{\tilde{u} - u, u} \notag \\
														&= \tilde{\Omega} \cdot d_{1/2} u + R_{1}^{\tilde{u}, u} + R_{2}^{\tilde{u}, u} + R_{3}^{\tilde{u}, u} + f
	\end{align}
	As for $N = S^{n-1}$, we may restrict our attention to the case of vanishing Fourier coefficients by removing the averages (i.e. $0$-th order Fourier coefficients) associated with each summand individually in \eqref{reglemmapg2caseequation}. This means, in order to render the fractional Laplacian invertible, we consider $v = u- \hat{u}(0)$ instead of $v = u$ below and remove averages for the contributions $R_1, R_2, R_3$ and $\Omega \cdot d_{1/2} v$. Let us observe that:
	\begin{align}
	\label{invertibleopgen}
		v \mapsto 		&(-\Delta)^{-1/2}\left( (-\Delta)^{1/2} v + (\tilde{\Omega} - \Omega) \cdot d_{1/2} v + R_{1}^{\tilde{u} - u, v} + R_{2}^{\tilde{u} - u, v} + R_{3}^{\tilde{u} - u, v} \right) \notag \\
					&= v + (-\Delta)^{-1/2}\left( (\tilde{\Omega} - \Omega) \cdot d_{1/2} v \right) + (-\Delta)^{-1/2}\left( R_{1}^{\tilde{u} - u, v} + R_{2}^{\tilde{u} - u, v} + R_{3}^{\tilde{u} - u, v} \right) 
	\end{align}
	defines an invertible mapping from $\dot{F}^{1/2}_{p,2}$ to itself, if we assume that $\tilde{\Omega}$ and $\tilde{u}$ are sufficiently good approximations. Use the estimates in Lemma \ref{refr1} and \ref{refr2} as well as the estimate:
	$$\| (\tilde{\Omega} - \Omega) \cdot d_{1/2} v \|_{L^{\frac{2p}{p+2}}} \lesssim \| \Omega - \tilde{\Omega} \|_{L^{2}_{od}} \| v \|_{\dot{F}^{1/2}_{p,2}},$$
	and the continuity of the embedding $L^{\frac{2p}{p+2}}(S^1) \hookrightarrow \dot{F}^{-1/2}_{p,2}(S^1)$ to deduce that the maps are sufficiently small in operatornorm, rendering \eqref{invertibleopgen} a small perturbation of the identity map and therefore invertible itself. One may also observe that the RHS of \eqref{maineqgeninv} lies in $L^{q}(S^1)$ for any $q < 2$ (using smoothness of the approximating terms), which is in the dual of $\dot{F}^{-1/2}_{p,2}$ thanks to Sobolev embeddings. This implies that there exists a unique solution $v \in \dot{F}^{1/2}_{p,2}$:
	\begin{align}
	\label{tosolve01}
		v + (-\Delta)^{-1/2} 	&\left( (\tilde{\Omega} - \Omega) \cdot d_{1/2} v + R_{1}^{\tilde{u} - u, v} + R_{2}^{\tilde{u} - u, v} + R_{2}^{\tilde{u} - u, v} \right) \notag \\
									&= (-\Delta)^{-1/2} \left( \tilde{\Omega} \cdot d_{1/2} u + R_{1}^{\tilde{u}, u} + R_{2}^{\tilde{u}, u} + R_{3}^{\tilde{u}, u} + f \right)
	\end{align}
	This shows that if $u \in \dot{F}^{1/2}_{p,2}(S^1)$ for some $p > 2$, it also lies in this space for any $q \geq 2$ by existence and uniqueness of solutions to the equation above due to the invertibility of the map and the natural inclusions $\dot{F}^{1/2}_{p,2} \hookrightarrow \dot{F}^{1/2}_{q,2}$ for every $p \geq q$, keeping in mind that the RHS of \eqref{tosolve01} is independent of $v$. This concludes the proof of Lemma \ref{casepg2}.
	\end{proof}
	
	\subsubsection{Small Initial Energy: Regularity Lemma in the Case $p = 2$}
	
	The missing step in order to be able to apply Lemma \ref{casepg2} to our energy class solution of the fractional harmonic gradient flow at a fixed time is provided by the following:
	
	\begin{lem}
	\label{casepe2}
		Let $u \in {H}^{1/2}(S^1) \cap L^{\infty}(S^1)$ and $f \in L^2(S^1)$ and assume that $u$ solves the non-local PDE:
		$$(-\Delta)^{1/2} u = d_{1/2} u \cdot d_{1/2} \left( d\pi^{\perp}(u) \right) + \div_{1/2} \left( \frac{A^{i}_{u}(du, du)(x,y)}{| x-y |^{1/2}} d\pi^{\perp}(u(y))_{ij} \right) + f,$$
		Then there exists $p > 2$, such that:
		$$u \in \dot{F}^{1/2}_{p,2}(S^1).$$
	\end{lem}
	
	Combining Lemma \ref{casepg2} and \ref{casepe2} yields the following corollary:
	
	\begin{cor}
	\label{mainreggeneral}
		Let $u \in {H}^{1/2}(S^1) \cap L^{\infty}(S^1)$ and $f \in L^2(S^1)$ and assume that $u$ solves:
		$$(-\Delta)^{1/2} u = d_{1/2} u \cdot d_{1/2} \left( d\pi^{\perp}(u) \right) + \div_{1/2} \left( \frac{A^{i}_{u}(du, du)(x,y)}{| x-y |^{1/2}} d\pi^{\perp}(u(y))_{ij} \right) + f,$$
		Then we know:
		$$u \in \dot{F}^{1/2}_{p,2}(S^1), \quad \forall p \geq 2.$$
	\end{cor}
	
	We observe that this corollary is essentially the analogue of \cite[Lemma 3.4]{wettstein} in the case $N = S^{n-1}$ and thus should suffice in order to conclude our investigation into the uniqueness of the fractional harmonic gradient flow with small initial energy. Indeed, we shall show that the same arguments carry over in the next subsection.
	
	\begin{proof}
	Let us turn to the proof of Lemma \ref{casepe2}: The technique we will be using relies on the ideas of \cite{daliopigati}, in particular \cite[Theorem D.7]{daliopigati} and \cite[Corollary D.8]{daliopigati}, we shall only point out the differences. Some more details are included in Appendix A for the reader's convenience:
	
	The main difference lies in the inclusion of a perturbation term in $L^2(S^1)$ and some minor changes to get higher local integrability for the fractional gradients by means of localisations. Indeed, the argument proceeds as in \cite[Theorem D.7]{daliopigati}, the main change being the aforementioned inclusion of the perturbation $f$ which leads to a summand of the form:
	$$\tilde{f} := \frac{f \circ \Pi^{-1}}{1+x^2} \in L^{q}(\R),$$
	for all $q \in [1,2]$, which is easily seen by using integrability of the function $1/(1+x^2)$ and the change of variables formula in combination with $f \in L^{2}(S^1)$. This term can be dealt with by solving the equation $(-\Delta)^{1/2} w_0 = f$ on the circle and using the stereographic projection to arrive at a solution of the corresponding problem on the real line, compare this with the proof of \cite[Lemma D.2]{daliopigati} for a more detailed argument.
	
	The pull-back of $w_0$ shall be denoted by $\tilde{w}_{0}$ and we see that $(-\Delta)^{1/4} \tilde{w}_{0}$ is then, since $(-\Delta)^{1/4} w_0 \in L^{p}(S^1)$ by Sobolev embeddings, locally in $L^p$ on the real line and lies in $L^{2}(\R)$ by the arguments in Claim 1 of the proof of Lemma D.2 in \cite{daliopigati}. In fact, since $w_0$ is in $H^1(S^1)$, we may immediately see, by using the chain rule, that:
	$$\tilde{w}_{0} \in \dot{W}^{1,p}(\R), \quad \forall 1 \leq p \leq 2$$
	This is a consequence of computing the derivative directly and observing the improved integrability provided by the multiplication by the derivative of the stereographic projection $1/(1+x^2)$. This also implies, by Sobolev embeddings, that 
	$$(-\Delta)^{1/4} \tilde{w}_{0} \in L^{q}(\R), \quad \forall 2 \leq q < +\infty$$
	Therefore, we have that $(-\Delta)^{1/4} \tilde{w}_{0}$ has sufficently good higher integrability for the arguments in Step 2 of Theorem D.7 of \cite{daliopigati} to work. The Morrey-estimate as in \cite{daliopigati} can be recovered by using $(-\Delta)^{1/4} \tilde{w}_{0}$ instead of $(-\Delta)^{-1/4}(Q\tilde{f})$ for an appropriate $Q$ as in \cite[Theorem D.7]{daliopigati}, with uniform bounds for all $x \in \R$, see also Appendix A for some of the key changes.

	Following the arguments in the proof of \cite[Lemma D.7]{daliopigati}, we arrive at the same conclusion in our scenario. In fact, if we continue on and keep arguing as in \cite[Corollary D.8]{daliopigati}, we find:
	$$(-\Delta)^{1/4} w \in L^{p}_{loc}(\R),$$
	for some $p > 2$, where $w$ is the pull-back of $u$ under the stereographic projection $\Pi$, i.e.:
	$$w := u \circ \Pi$$
	This follows by the conclusion of \cite[Theorem D.7]{daliopigati} together with Adam's embedding, as well as some arguments relating to the Riesz operator, see \cite[Corollary D.8]{daliopigati}. Thus, using a compactly supported $\rho$ which is equal to $1$ on some ball $B$, we introduce:
	$$\tilde{w} := (-\Delta)^{-1/4} \left( \rho (-\Delta)^{1/4} w \right)$$
	Observing that $w - \tilde{w}$ is $1/4$-harmonic on $B$, we may use the fractional mean value property (see for example \cite{bucur}) of the $1/4$-Laplacian or the associated regularity theory (see for example \cite{jeonpetro} and the references therein, the results being applicable due to the integrability properties of $w$ and $\tilde{w}$) to deduce that:
	$$w - \tilde{w} \text{ is at least $C^1$ on } B_{3/4},$$
	with $B_{3/4}$ being the ball with the same center as $B$, but $3/4$ of the radius of $B$. Here we actually have to make use of $w \in L^{\infty}$, keeping in mind that $w = u \circ \Pi$, as well as $\tilde{w} \in L^{q}(\R)$ for some $1 < q < 2$ by Hardy-Littlewood-Sobolev's inequality to guarantee that the mean value integrals exist. This allows us now to estimate:
	$$\left( \int_{B_{1/2}} \left( \int_{B_{1/2}} \frac{| w(x) - w(y) |^2}{| x-y |^2} dy \right)^{p/2} dx \right)^{1/p} < \infty,$$
	simply by observing that $\tilde{w}$ satisfies such an estimate and using smoothness of the difference $w - \tilde{w}$ as well as compactness of the ball combined with the usual triangle inequality. Using the stereographic projection to pull back and estimating the emerging Jacobians on the compact subsets, the same holds on the circle $S^1$. Using $B$ large enough and arguing the same way for the stereographic projection at other points and taking a finite covering of $S^1$ by such balls, we find that:
	$$\left( \int_{S^{1}} \left( \int_{S^{1}} \frac{| u(x) - u(y) |}{| x-y |^2} dy \right)^{p/2} dx \right)^{1/p} < \infty,$$
	for some $p > 2$. This is precisely the desired result.
	\end{proof}
	
	We emphasise that the proof implicitely, by referring to the techniques in \cite{daliopigati}, and heavily relies on commutator estimates and hence on compensation phenomena in the non-local setting. This dependence is somewhat more obvious in \cite{wettstein} and $N= S^{n-1}$ where we apply the fractional Wente-type estimate from \cite{mazoschi} directly much in the same way \cite{riv} uses the classical Wente estimate.
	
	\subsubsection{Small Initial Energy: Application of the Regularity Lemmas}
	
	We are finally able to apply the results from the previous subsections to the case of an energy class solution $u$ of the $1/2$-harmonic gradient flow with values in $N \subset \R^n$ a closed manifold. The steps are completely analogous to the case $N = S^{n-1}$ and \cite{riv}, cf. \cite{wettstein}. We will be proving the following as the main result, which we have already mentioned in the introductionas part of Theorem \ref{mainresultsection1}:
	
	\begin{thm}
	\label{uniqueinenergygeneral}
		Let $u: \R_{+} \times S^1 \to N \subset \R^n$, $N$ a closed and smooth manifold, be a solution of the weak fractional harmonic gradient flow \eqref{gradflowuvgeneral} with initial datum $u_0 \in H^{1/2}(S^1;N)$ and satisfying the following regularity assumptions:
		$$u \in L^{\infty}(\R_{+}; H^{1/2}(S^1)); \quad u_t \in L^{2}(\R_{+}; L^{2}(S^{1}))$$
		Then there exists $\varepsilon > 0$, such that among all such $u$ satisfying the smallness condition:
		$$\| ( - \Delta)^{1/4} u(t) \|_{L^2(S^1)} \leq \varepsilon, \quad \forall t \in \R_{+},$$
		the solution to the fractional harmonic gradient flow \eqref{gradflowuvgeneral} with initial datum $u_0$ is unique.
	\end{thm}
	
	In particular, if we assume that the $1/2$-energy:
	$$E_{1/2}(u(t)) := \frac{1}{2} \| (-\Delta)^{1/4} u(t) \|_{L^{2}(S^1)}^{2},$$
	is non-increasing in time, as motivated in the case $N = S^{n-1}$ by \cite[Lemma 3.3]{wettstein}, the smallness condition could be rephrased as:
	$$\| ( - \Delta)^{1/4} u_0 \|_{L^2(S^1)} \leq \varepsilon$$
	An important first step in the proof of Theorem \ref{uniqueinenergygeneral} is the following:
	
	\begin{prop}
		Let $u: \R_{+} \times S^1 \to N \subset \R^n$ satisfy the following regularity assumptions:
		$$u \in L^{\infty}(\R_{+}; H^{1/2}(S^1)); \quad u_t \in L^{2}(\R_{+}; L^{2}(S^{1}))$$ 
		Moreover, assume $u$ solves the half-harmonic gradient flow equation \eqref{gradflowuvgeneral}. Then for almost every time $t > 0$, we have:
		$$u(t) \in H^{1}(S^1)$$
	\end{prop}
	
	The proof is more or less an immediate application of Corollary \ref{mainreggeneral}, once we observe that $u \in N$ for almost every $(t,x) \in \R_{+} \times S^1$ implies $u(t) \in L^{\infty}(S^1)$ for almost every time $t$.
	
	\begin{proof}
	First, by noticing that we may apply Corollary \ref{mainreggeneral} with $p = 4$, we notice that $u(t) \in H^1(S^1)$ for almost every $t \in \mathbb{R}$, since the RHS of the rephrasing of the fractional harmonic flow \eqref{rephrasing01} below is in $L^{2}$ and the Riesz potential preserves the $L^2$-norm. Indeed, we see for a fixed time $t$:
	\begin{equation}
	\label{rephrasing01}
		(-\Delta)^{1/2} u(t) = d_{1/2} u \cdot d_{1/2} \left( d\pi^{\perp}(u) \right) + \div_{1/2} \left( \frac{A^{i}_{u}(du, du)(x,y)}{| x-y |^{1/2}} d\pi^{\perp}(u(y))_{ij} \right) - \partial_{t} u(t)
	\end{equation}
	Use the remainder estimates in Lemma \ref{refr1} and the bounds on $B_{u}(x,y)$ to find that the first two summands lie in $L^{2}(S^1)$, see \eqref{estr1reg} and the estimate in Theorem \ref{schiwangthm1.4} for $S^1$ with $p=4, q=2, s=1/2$ for the first summand in \eqref{rephrasing01}. Hence, by standard elliptic estimates for the fractional Laplacian or simply observing that with $\mathcal{R}$ being the Riesz transform, we have:
	\begin{equation}
		\nabla u(t) = \mathcal{R} \left( d_{1/2} u \cdot d_{1/2} \left( d\pi^{\perp}(u) \right) + \div_{1/2} \left( \frac{A^{i}_{u}(du, du)(x,y)}{| x-y |^{1/2}} d\pi^{\perp}(u(y))_{ij} \right) - \partial_{t} u(t) \right), 
	\end{equation}
	and $\mathcal{R}$ being a continuous operator on $L^{2}(S^1)$, we are led to the following estimate:
	\begin{align}
	\label{estforh11}
		\| u(t) \|_{H^{1}(S^1)}^2 	&\leq C \left( \| u(t) \|_{L^2}^2 + \| | d_{1/2} u(t) |^{2} \|_{L^2}^2 + \| \partial_{t} u(t) \|_{L^2}^2 \right) \notag \\
							&\leq C \left( 1 + \| | d_{1/2} u(t) |^{2} \|_{L^2}^2 + \| \partial_{t} u(t) \|_{L^2}^2 \right),
	\end{align}
	where we used $u(t) \in N$ almost everywhere for almost every time $t$ to bound the $L^2$-norm of $u(t)$ by an $L^{\infty}$-bound depending only on $N$. To provide some more details, the first two summands on the RHS of the equation \eqref{rephrasing01} may be estimated by the $L^2$-norm of $| d_{1/2} u |^2(x)$, see Lemma \ref{refr1} in the second case with $u= v = w$ and the treatment of uniqueness for a rephrasing of $d_{1/2} u \cdot d_{1/2} \left( d\pi^{\perp}(u) \right)$ in terms of $B_{u}(x,y)$ which is bounded. We highlight that:
	$$\| | d_{1/2} u(t) |^{2} \|_{L^2}^2 \sim \| u(t) \|_{\dot{F}^{1/2}_{4,2}}^{2},$$
	by Theorem \ref{schiwangthm1.4}, which together with Corollary \ref{mainreggeneral} in the case $p=4$ allows us to conclude that $u(t) \in H^{1}(S^1)$.
	\end{proof}
	
	It is clear that thanks to the (local) $L^2$-integrability with respect to time, it thus remains to study the following contribution:
	$$\| | d_{1/2} u(t) |^{2} \|_{L^2}^2 = \| u \|_{\dot{W}^{1/2, (4,2)}}^4 \sim \| u \|_{\dot{F}^{1/2}_{4,2}}^4 \sim \| (-\Delta)^{1/4} u \|_{L^4}^4$$
	Using the same ideas as in the proof of the uniqueness statement for $N = S^{n-1}$ (see \cite{wettstein}), we may estimate this using some cut-off function applied to the periodic extension of $u$ or using \cite[Lemma 3.1]{wettstein} and Theorem \ref{schiwangthm1.4} by:
	$$\| | d_{1/2} u(t) |^{2} \|_{L^2}^2 \leq C' \| u(t) \|_{\dot{H}^{1/2}}^{2} \| u(t) \|_{\dot{H}^{1}}^{2}$$
	We recall that the homogeneous norm may be used by means of a perturbation-argument using constants as in the case $N = S^{n-1}$. Therefore, we have the energy term appearing as for $N = S^{n-1}$ and provided it is smaller than some $\varepsilon > 0$, we find:
	$$\| | d_{1/2} u(t) |^{2} \|_{L^2}^2 \leq C' \| (-\Delta)^{1/4}u(t) \|_{L^{2}}^{2} \| u(t) \|_{H^{1}}^{2} \leq C' \varepsilon \cdot \| u(t) \|_{H^{1}}^{2},$$
	where $\varepsilon > 0$ is an a priori energy estimate as in \cite{riv}. If $\varepsilon > 0$ is sufficiently small, we may absorb this term in the left hand side of \eqref{estforh11} to arrive at:
	\begin{equation*}
		(1 - CC' \varepsilon ) \cdot \| u(t) \|_{H^{1}(S^1)}^2 \leq \tilde{C} \left( 1 + \| \partial_{t} u(t) \|_{L^2}^2 \right) \Rightarrow \| u(t) \|_{H^{1}(S^1)} \leq \frac{\tilde{C}}{1 - C' C \varepsilon} \left( 1 + \| \partial_{t} u(t) \|_{L^2}^2 \right),
	\end{equation*}
	which thus yields an estimate for the $H^1$-norm, if, for example, $0 < \varepsilon \leq 1/(2C' C)$. We observe that hence, by the integrability properties of $\partial_{t} u$ and the constant function:
	\begin{equation}
	\label{improvedregularityinh1forgeneral}
		u \in L^{2}_{loc}(\R_{+}; H^{1}(S^1))
	\end{equation}
	This allows us to employ Theorem \ref{uniquegeneral}, as the local regularity of the derivative is now apparent. This proves precisely the result in Theorem \ref{uniqueinenergygeneral} which is also stated in the introduction of the paper and thus concludes our investigation of the fractional harmonic gradient flow with small initial energy.

	\subsection{Existence and Regularity of Solutions}
	
	It remains to establish that solutions of the fractional harmonic gradient flow exist and are smooth, at least locally, and even globally smooth, provided the initial energy is sufficiently small. The ideas behind the proof are mostly the same as in \cite{wettstein}, once we have rewritten the fractional harmonic gradient flow in a slightly different way. It should be notes that the existence result we prove is extending the one in \cite{schisirewang} which only deals with certain closed manifolds $N$ that possess nice symmetry properties.

	\subsubsection{An Equivalent Reformulation of the Main Equation}
	
	A key step in \cite{wettstein} in order to derive local existence lies in the application of the Inverse Function Theorem in Banach spaces to argue along the lines of \cite{hamilton}. As one can see in the author's previous work \cite{wettstein}, the property that the solution assumes values only in $S^{n-1}$ is merely proven after establishing existence and therefore crucially relies on the fact that $u \in N$ is ensured by the $1/2$-harmonic gradient flow, provided the initial datum takes values in $N$. It is thus reasonable to expect that we shall treat the target space after establishing local existence. Nevertheless, the choice of formulation of the fractional harmonic gradient flow we study will be of great importance when it comes to verifying $u \in N$ a.e.. As a result, we first would like to think about the right kind of equation to study.\\
	
	First, in \cite{wettstein} we used the following sequence of equivalent characterisations:
	$$u(t,x) \in S^{n-1} \Leftrightarrow | u(t,x) |^2 = 1 \Leftrightarrow | u(t,x) |^2 - 1 = 0$$
	Unfortunately, quite such a simple characterisation is not available for general $N$. However, if we let $\pi: \mathbb{R}^n \to \mathbb{R}^n$ be the extended version of the closest point projection, see Section 4.1, we may see for $u$ at least continuous and $u(0) = u_0$ smooth with values in $N$:
	$$u(t,x) \in N \Leftrightarrow \pi(u(t,x)) = u(t,x) \Leftrightarrow | u(t,x) - \pi(u(t,x)) |^2 = 0,$$
	for all $(t,x) \in [0,\infty[ \times S^1$. The key observation is that due to the continuity and the fact that $\pi(x) = x$ only on $N$ and possibly on a subset of the complement of a sufficiently small neighbourhood of $N$, the identity $\pi(u(t,x)) = u(t,x)$ for all $(t,x)$ actually necessitates $u(t,x) \in N$, as $u(0,x) \in N$ for all $x \in S^1$. The minimal regularity imposed by assuming $u$ is continuous actually prevents $u$ from ever leaving $N$, since the set of fixed points of $\pi$ is a disconnected union of $N$ and some other set. Thus, even though the extension $\pi$ is not canonical, the condition:
	$$| u(t,x) - \pi(u(t,x)) |^2 = 0, \quad \forall (t,x) \in [0,\infty[ \times S^1,$$
	is the analogue we are looking for to the function $| u |^2 - 1$ in the case $N = S^{n-1}$.\\
	
	Now, we would like to think about the fractional heat-type equation solved by $| u - \pi(u) |^2$. This will provide us crucial information about the "correct" choice of non-linearity to study in connection with the $1/2$-harmonic gradient flow. So we are interested in computing:
	$$\partial_{t} \left( | u - \pi(u) |^2 \right) + (-\Delta)^{1/2} \left( | u - \pi(u) |^2 \right)$$
	For now, we assume that $u$ is actually smooth to justify our calculations. Then:
	$$\partial_{t} \left( | u - \pi(u) |^2 \right) = 2 \left( u_{t} - d\pi(u) u_{t} \right) \cdot (u - \pi(u))$$
	Completely analogous to the computations for local existence and regularity in \cite{wettstein}, we have:
	$$(-\Delta)^{1/2} \left( | u - \pi(u) |^2 \right) = 2 (-\Delta)^{1/2} \left( u - \pi(u) \right) \cdot (u - \pi(u)) - | d_{1/2} \left( u - \pi(u) \right) |^2$$
	Therefore:
	\begin{align}
	\label{exreg:eq001}
		&\partial_{t} \left( | u - \pi(u) |^2 \right) + (-\Delta)^{1/2} \left( | u - \pi(u) |^2 \right)	\notag \\
		&= 2 \left( u_{t} - d\pi(u) u_{t} \right) \cdot (u - \pi(u)) + 2 (-\Delta)^{1/2} \left( u - \pi(u) \right) \cdot (u - \pi(u)) - | d_{1/2} \left( u - \pi(u) \right) |^2 \notag \\
		&= 2 \left( Id - d\pi(u) \right) \left( u_{t} + (-\Delta)^{1/2} u \right) \cdot (u - \pi(u)) \notag \\
		&+ 2 \left( d\pi(u) (-\Delta)^{1/2} u - (-\Delta)^{1/2} (\pi(u)) \right) \cdot (u - \pi(u)) - | d_{1/2} \left( u - \pi(u) \right) |^2 \notag \\
		&= 2 (Id - d\pi(u)) r(u) \cdot (u - \pi(u)) + 2 \left( d\pi(u) (-\Delta)^{1/2} u - (-\Delta)^{1/2} (\pi(u)) \right) \cdot (u - \pi(u)) \notag \\
		&- | d_{1/2} \left( u - \pi(u) \right) |^2,
	\end{align}
	where we write:
	\begin{equation}
	\label{exreg:eq0011}
		u_t + (-\Delta)^{1/2} u = r(u),
	\end{equation}
	with $r$ the non-linearity depending on $u$ we are trying to find. Thinking about \eqref{exreg:eq001}, one may come up with the following non-linearity:
	\begin{equation}
	\label{exreg:eq002}
		r(u) := C(d\pi(u), u),
	\end{equation}
	where we define:
	\begin{equation}
	\label{exreg:eq003}
		C(a,b) := \mathcal{R} (a \nabla b) - a (-\Delta)^{1/2} b
	\end{equation}
	Here, we denote by $\mathcal{R}$ the Riesz-Hilbert transform on $S^1$. Naturally, \eqref{exreg:eq003} extends to vector-valued or matrix-valued maps in the natural way. Additionally the operator is already studied in \cite{daliopigati} and one of the results there, which easily translates to $S^1$ by the same arguments, is the following:
	
	\begin{prop}[Lemma E.2, \cite{daliopigati}]
	\label{lemmae.2daliopigati}
		Assume that $a \in F^{s}_{p,2}(S^1), b \in F^{1}_{q,2}(S^1)$ with $s > 1/p, 1 < p,q < \infty$. Then, for any $\gamma > 1/p$, we have the following estimate:
		$$\| C(a,b) \|_{F^{s-\gamma}_{q,2}(S^1)} \lesssim \| a \|_{F^{s}_{p,2}(S^1)} \| b \|_{F^{1}_{q,2}(S^1)}$$
	\end{prop}
	
	The proof only relies on the characterisations of the Bessel-Sobolev spaces and the use of Littlewood-Paley decompositions, both of which continue to hold on $S^1$ by \cite{schmeitrieb} and our discussion in Section 2. We thus refer to \cite{daliopigati} for the proof.\\
	
	Observe that Proposition \ref{lemmae.2daliopigati} actually also hints at nice bootstrapping estimates available for $C(a,b)$ and therefore, solutions of \eqref{exreg:eq0011} can be expected to be smooth, provided the initial datum is smooth as well, compare this with the ideas in \cite{wettstein}. However, for the moment, we would like to explore how this choice of $r(u)$ affects the computation in \eqref{exreg:eq001}:
	\begin{align}
	\label{exreg:eq004}
		&\partial_{t} \left( | u - \pi(u) |^2 \right) + (-\Delta)^{1/2} \left( | u - \pi(u) |^2 \right)	\notag \\
		&= 2 (Id - d\pi(u)) r(u) \cdot (u - \pi(u)) + 2 \left( d\pi(u) (-\Delta)^{1/2} u - (-\Delta)^{1/2} (\pi(u)) \right) \cdot (u - \pi(u)) - | d_{1/2} \left( u - \pi(u) \right) |^2 \notag \\
		&= 2 (Id - d\pi(u)) \left( (-\Delta)^{1/2}(\pi(u)) - d\pi(u) (-\Delta)^{1/2} u \right) \cdot (u - \pi(u)) \notag \\
		&+ 2 \left( d\pi(u) (-\Delta)^{1/2} u - (-\Delta)^{1/2} \left( \pi(u) \right) \right) \cdot (u - \pi(u)) - | d_{1/2} \left( u - \pi(u) \right) |^2 \notag \\
		&= - 2 d\pi(u) \left( (-\Delta)^{1/2}(\pi(u)) - d\pi(u) (-\Delta)^{1/2} u \right) \cdot (u - \pi(u)) - | d_{1/2} (u - \pi(u)) |^2 \notag \\
		&= 2 \left( d\pi(\pi(u)) - d\pi(u) \right) \left( (-\Delta)^{1/2}(\pi(u)) - d\pi(u) (-\Delta)^{1/2} u \right) \cdot (u - \pi(u)) - | d_{1/2} (u - \pi(u)) |^2,
	\end{align}
	where in the last line, we used that $u - \pi(u)$ is orthogonal to $T_{\pi(u)} N$ and thus:
	$$d\pi(\pi(u))(u - \pi(u)) = 0.$$
	Applying the fact that $d\pi(\pi(u))$ is an orthogonal projection and therefore symmetric, we may deduce the equality above in \eqref{exreg:eq004}. Using now Lipschitz-continuity of $\pi$ and its derivatives as well as smoothness, we may therefore show:
	\begin{equation}
	\label{innbyniceestimate}
		\partial_{t} \left( | u - \pi(u) |^2 \right) + (-\Delta)^{1/2} \left( | u - \pi(u) |^2 \right) \leq C \cdot | u - \pi(u) |^2,
	\end{equation}
	where $C$ is a constant depending on $u$. We expand a bit on this step in the next subsection when determining the kernel of the linearisation of the operator induced by the fractional harmonic gradient flow. Since $u - \pi(u) = 0$ at time $t=0$, we may therefore invoke a maximum principle inspired by the one in \cite{hamilton} just like in \cite{wettstein} to deduce:
	$$| u - \pi(u) |^2 = 0, \quad \forall (t,x) \in [0, \infty[ \times S^1,$$
	and consequently:
	$$u(t,x) \in N, \quad \forall (t,x) \in [0, \infty[ \times S^1$$
	We notice that this precisely proves the required condition on the values of $u$. Additionally, the argument works equally well on $[0,T] \times S^1$, thus it applies also to local solutions of \eqref{exreg:eq0011} in time. This observation will be crucial, as it will allow us to forget about the condition on the values at first and focus on the analytic aspects of the PDE.
	
	\subsubsection{Local Regularity}
	
	Similar to \cite{wettstein}, we shall prove the following result dealing with local regularity of solutions:
	
	\begin{prop}
	\label{exreg:proplocalreg}
		Let $u_0 \in C^{\infty}(S^1;N)$. Then there exists a $T > 0$, possibly depending on $u_0$, and a smooth map $u \in C^{\infty}([0,T] \times S^1)$ solving the following non-local PDE:
		\begin{equation}
		\label{exreg:eq005}
			u_t + (-\Delta)^{1/2} u = (-\Delta)^{1/2} \pi(u) - d\pi(u) (-\Delta)^{1/2} u,
		\end{equation}
		and satisfying $u(0, \cdot) = u_0$. Additionally, by the result in the previous subsection:
		$$u(t,x) \in N, \quad \forall (t,x) \in [0,T] \times S^1,$$
		and, as a result, \eqref{exreg:eq005} becomes the half-harmonic gradient flow equation \eqref{gradflowuvgeneral}.
	\end{prop}
	
	The last observation is due to the fact that if $u \in N$, then $d\pi(u)$ is the projection onto the tangent space $T_{u} N$ and $\pi(u) = u$, i.e.:
	$$(-\Delta)^{1/2} \pi(u) - d\pi(u) (-\Delta)^{1/2} u = (-\Delta)^{1/2} u - d\pi(u) (-\Delta)^{1/2} u = (Id - d\pi(u)) (-\Delta)^{1/2} u,$$
	which implies that the RHS of \eqref{exreg:eq005} is actually orthogonal to $T_{u} N$. This is pecisely the meaning of \eqref{gradflowuvgeneral}, see also the computation in Section 3.1, Section 4.1 and \cite[Section 3.1]{wettstein}.\\
	
	\noindent
	\textit{Proof of Proposition \ref{exreg:proplocalreg}.} The proof actually goes along the very same lines as the proof of Proposition 3.2 in \cite{wettstein}, using the local Inversion Theorem for Banach spaces and slightly better integrability combined with a bootstrap procedure which now uses Proposition \ref{lemmae.2daliopigati} instead of the Lemma proven in \cite{wettstein}.\\
	
	As seen in \cite{wettstein} by Fourier representation, we may find $\tilde{u}$ solving the fractional heat equation:
	$$\tilde{u}_{t} + (-\Delta)^{1/2} \tilde{u} = 0, \quad \tilde{u}(0,\cdot) = u_0$$
	By the same argument using an explicit formula for $\tilde{u}$, we know that $\tilde{u}$ is actually smooth, as $u_0$ is smooth. Let us now define the following map for $1 < p < \infty$:
	\begin{align}
	\label{exreg:eq006}
		&H: W^{1,p}_{0}([0,T] \times S^1) \to L^{p}( [0,T] \times S^1) \notag \\
		&H(v) := (\tilde{u} + v)_{t} + (-\Delta)^{1/2} \left( \tilde{u} + v \right) - C(d\pi(\tilde{u} + v), \tilde{u} + v)
	\end{align}
	Here, $u \in W^{1,p}_{0}([0,T] \times S^1) \subset W^{1,p}([0,T] \times S^1)$ denotes the subspace of functions in $W^{1,p}([0,T] \times S^1)$ with $u(0,\cdot) = 0$. We observe that if $H(\tilde{u} + v)$ is vanishing on some subinterval $[0,T_0] \subset [0,T]$, then $\tilde{u} + v$ is actually a local solution to the half-harmonic gradient flow \eqref{exreg:eq005}. Therefore, it suffices to establish the existence of $v$ with this property. To achieve this, as in \cite{wettstein}, we will show that $H$ maps a sufficiently small neighbourhood of the zero function to an open neighbourhood of $H(\tilde{u})$ and then choose $f \in L^{p}([0,T] \times S^1)$ such that $f$ equals $0$ on an interval $[0,T_0]$ and agreeing with $H(\tilde{u})$ on the remainder of $[0,T]$. Choosing $T_0$ sufficiently small, $f$ then lies in the image of $H$ and thus a function $v$ with the desired properties exists.\\
	
	To follow the program outlined above, we would like to invoke the inverse function theorem for Banach spaces. Let us observe that $H$ is Fr\`echet-differentiable and:
	$$dH(0) h = h_{t} + (-\Delta)^{1/2} h - C \left( d\left( d\pi \right)(\tilde{u}) h, \tilde{u}\right) - C \left( d\pi(\tilde{u}),  h \right)$$
	If we are able to show that $dH(0)$ is invertible, then the inverse function theorem would apply and we may argue as previously stated. Using Theorem 3.1 in \cite{hieber}, we deduce that:
	$$h \mapsto h_{t} + (-\Delta)^{1/2} h,$$
	on the function space above is actually invertible. Therefore, if we can establish that the remaining summand:
	\begin{equation}
	\label{exreg:eq007}
		h \mapsto C \left( d\left( d\pi \right)(\tilde{u}) h, \tilde{u} \right) + C \left( d\pi(\tilde{u}),  h \right),
	\end{equation}
	is compact, then $dH(0)$ would be Fredholm and thus:
	$$dH(0) \text{ is invertible } \Leftrightarrow dH(0) \text{ is injective }$$
	Indeed, one may observe that (by using $\nabla (ab) = \nabla a \cdot b + a \cdot \nabla b$ and the definition of $\mathcal{R}$):
	$$C(a,b) = (-\Delta)^{1/2} a \cdot b - \mathcal{R} \left( \nabla a \cdot b \right) - 2 d_{1/2} a \cdot d_{1/2} b,$$
	which shows that the second summand in \eqref{exreg:eq007} is a compact map. Namely, we have:
	\begin{align}
	\label{alternativeformofcommutator}
		C(a,b)	&= \mathcal{R} (a \nabla b) - a (-\Delta)^{1/2} b \notag \\
				&= \mathcal{R}( \nabla (ab)) - \mathcal{R}(\nabla a \cdot b) - a (-\Delta)^{1/2} b \notag \\
				&= (-\Delta)^{1/2} (ab) - (-\Delta)^{1/2} a \cdot b - a (-\Delta)^{1/2} b - \mathcal{R}(\nabla a \cdot b) + (-\Delta)^{1/2} a \cdot b \notag \\
				&= - 2 d_{1/2} a \cdot d_{1/2} b - \mathcal{R}(\nabla a \cdot b) + (-\Delta)^{1/2} a \cdot b,
	\end{align}
	where we used in the last equality the singular integral formulations of the fractional Laplacian. Reordering now provides the desired formula. For the first summand in \eqref{exreg:eq007}, we may just use Proposition \ref{lemmae.2daliopigati} and compactness of Sobolev embeddings. In both cases, we may easily deal with all terms involving $\tilde{u}$ by using the smoothness of this function. Thus, the operator $dH(0)$ is Fredholm. Since $dH(0)$ consists of an invertible operator and a compact one and adding a compact operator does not affect the Fredholm index, which means that $dH(0)$ has index $0$. As a result, invertibility of $dH(0)$ becomes equivalent to injectivity.\\
	
	Next, we want to investigate the kernel of $dH(0)$ to ultimately show that the linearised operator $dH(0)$ is injective. We argue by contradiction, i.e., assume that $h$ is such that:
	$$dH(0) h = 0.$$
	Using Proposition \ref{lemmae.2daliopigati}, we now may deduce that $h \in C^{\infty}([0,T] \times S^1)$. Namely, we observe that the summand $C  \left( d\pi(\tilde{u}),  h \right)$ is arbitrarily regular with respect to $x$ by using Proposition \ref{lemmae.2daliopigati}. Integrability in $L^p$ follows, as we may uniformly estimate all terms involving $\tilde{u}$ due to smoothness. For the second summand, we may argue analogous to \cite{wettstein}: By Proposition \ref{lemmae.2daliopigati}, we see:
	\begin{align}
		\| (-\Delta)^{s/2 - 1/4} C( d(d\pi)(\tilde{u})h, \tilde{u}) \|_{L^{p}(S^1)} 	&\lesssim \| d(d\pi(\tilde{u})) h \|_{F^{s}_{p,2}(S^1)} \| \nabla \tilde{u} \|_{L^{p}(S^1)} \notag \\
															&\lesssim \| h \|_{L^{p}(S^1)} + \| (-\Delta)^{s/2} h \|_{L^{p}(S^1)},
	\end{align}
	where the inequality involves constants depending on $\tilde{u}$, and therefore also:
	$$\| (-\Delta)^{s/2 - 1/4} C( d(d\pi)(\tilde{u})h, \tilde{u}) \|_{L^{p}([0,T] \times S^1)} \lesssim \| h \|_{L^{p}([0,T] \times S^1)} + \| (-\Delta)^{s/2} h \|_{L^{p}([0,T] \times S^1)}$$
	Let us now see the following for any $\varphi \in C^{\infty}([0,T] \times S^1)$ with compact support strictly contained in $[0,T] \times S^1$. Then we have:
	\begin{align}
		&\int_{0}^{T} \int_{S^1} (-\Delta)^{s/2} h \cdot (- \partial_{t} + (-\Delta)^{1/2} ) \varphi dx dt \notag \\
		&= \int_{0}^{T} \int_{S^1} h \cdot (- \partial_{t} + (-\Delta)^{1/2} ) (-\Delta)^{s/2} \varphi dx dt \notag \\
		&= \int_{0}^{T} \int_{S^1} \left( C(d(d\pi)(\tilde{u}) h, \tilde{u}) + C(d\pi(\tilde{u}), h) \right) (-\Delta)^{s/2} \varphi dx dt + \int_{S^1} h(0,x) (-\Delta)^{s/2} \varphi(0,x) \notag \\
		&= \int_{0}^{T} \int_{S^1} (-\Delta)^{s/2} \left( C(d(d\pi)(\tilde{u}) h, \tilde{u}) + C(d\pi(\tilde{u}), h) \right) \varphi dx dt
	\end{align}
	where we observed that $(-\Delta)^{s/2} \varphi$ is still compactly supported and smooth as well as the equation solved by $h$ and the initial condition $h(0,\cdot) = 0$. Therefore, $(-\Delta)^{s/2} h$ solves an inhomogeneous fractional heat equation with RHS in $L^{p}$. Arguing by using the connection of $\partial_{t} + (-\Delta)^{1/2}$ with the Laplacian and using an extension:
	$$\tilde{h}(t,x) := \begin{cases} h(t,x), \quad &\text{ if } t \geq 0 \\ - h(-t,x), \quad &\text{ if } t \leq 0 \end{cases}$$
	by symmetrising of $h$ to times $t \in [-T,0]$ (the resulting map $(-\Delta)^{s/2} \tilde{h}$ solves a Laplace equation with RHS in $W^{-1,p}$ in the distributional sense and we may thus argue by elliptic regularity), we find that $(-\Delta)^{s/2} h \in W^{1,p}_{loc}([0,T] \times S^1)$. Notice that as $h(0) = 0$, the extension $\tilde{h}$ is well-behaved. Arguing as in \cite{wettstein} using \cite{hieber}, where existence and uniqueness of solutions to the half-harmonic gradient flow in appropriate Sobolev spaces is implicitly treated, we may thus deduce:
	\begin{equation}
	\label{slightreggainforh}
		(-\Delta)^{s/2} h \in W^{1,p}([0,T] \times S^1), \quad \forall s \in [0,3/4],
	\end{equation}
	by using the estimate in Proposition \ref{lemmae.2daliopigati} as specified before. The next step is to actually find an equation solved by $\partial_{x} h =: h'$. This is achieved by using a slightly modified version of the equation $dH(0) h = 0$ using \eqref{alternativeformofcommutator}. Namely, we use:
	\begin{equation}
	\label{equationforhtobedifferentiated}
		h_{t} + (-\Delta)^{1/2} h = C(d(d\pi(\tilde{u})h, \tilde{u}) - 2 d_{1/2} \left( d\pi(\tilde{u}) \right) \cdot d_{1/2} h - \mathcal{R} \left( \nabla d\pi(\tilde{u}) \cdot h \right) + (-\Delta)^{1/2} \left( d\pi(\tilde{u}) \right) \cdot h
	\end{equation}
	If we differentiate both sides with respect to $x$, this leads to:
	\begin{align}
	\label{equationforhafterdifferentiating}
		\left( \partial_{t} + (-\Delta)^{1/2} \right) h' 	&=  C(\left(d(d\pi(\tilde{u}) \right)' h + d(d\pi(\tilde{u})) h', \tilde{u}) \notag \\
										&- 2 d_{1/2} \left( d\pi(\tilde{u}) \right)' \cdot d_{1/2} h - 2 d_{1/2} \left( d\pi(\tilde{u}) \right) \cdot d_{1/2} h' \notag \\
										&- \mathcal{R} \left( \nabla \left( d\pi(\tilde{u}) \right)' \cdot h \right) - \mathcal{R} \left( \nabla d\pi(\tilde{u}) \cdot h' \right) \notag \\
										&+ (-\Delta)^{1/2} \left( d\pi(\tilde{u}) \right)' \cdot h + (-\Delta)^{1/2} \left( d\pi(\tilde{u}) \right) \cdot h'
	\end{align}
	A direct computation using \eqref{slightreggainforh}, we thus know that the RHS of the equation lies in $L^{p}([0,T] \times S^1)$. Arguing as before using distributional solutions for the symmetrisation and the connection to the Laplacian as well as \cite{hieber}, we deduce:
	$$h' \in W^{1,p}([0,T] \times S^1])$$
	Inserting this into the main equation $dH(0)h = 0$, we thus find by differentiating with respect to $t$:
	$$h \in W^{2,p}([0,T] \times S^{1})$$
	By iterating similar to \cite{wettstein} and our computations starting from \eqref{equationforhafterdifferentiating} as above and repeating the same steps for higher and higher derivatives, using the connection to the Laplacian and Theorem 3.1 in \cite{hieber} repeatedly, this shows:
	$$\forall s \in \mathbb{R}_{\geq 0}: (-\Delta)^{s} h \in W^{1,p}([0,T] \times S^1)$$
	Using the equation $dH(0) h = 0$, we may also discover estimates for higher order derivatives in $t$-direction. Hence:
	$$h \in \bigcap_{k \in \mathbb{N}} W^{k,p}([0,T] \times S^1) \subset C^{\infty}([0,T] \times S^1)$$
	Therefore, any $h$ in the kernel of $dH(0)$ is actually smooth.\\
	
	It remains to establish that actually $h = 0$. The trick is as in \cite{wettstein} and the argument presented actually provides the argument for the missing step to prove \eqref{innbyniceestimate} in Section 4.3.1: We study the fractional heat-type equation satisfied by $| h |^2$. One finds:
	\begin{equation}
	\label{exreg:eq008}
		\partial_{t} \left( | h |^2 \right) = 2 h_{t} \cdot h,
	\end{equation}
	as well as:
	\begin{equation}
	\label{exreg:eq009}
		(-\Delta)^{1/2} \left( | h |^2 \right) = 2 (-\Delta)^{1/2} h \cdot h - | d_{1/2} h |^2.
	\end{equation}
	Combining \eqref{exreg:eq008} and \eqref{exreg:eq009}, we find:
	\begin{align}
	\label{exreg:eq010}
		&\partial_{t} \left( | h |^2 \right) + (-\Delta)^{1/2} \left( | h |^2 \right) \notag \\
		&= 2 h_{t} \cdot h + 2 (-\Delta)^{1/2} h \cdot h - | d_{1/2} h |^2 \notag \\
		&= 2 \left( h_{t} + (-\Delta)^{1/2} h \right) \cdot h - | d_{1/2} h |^2 \notag \\
		&= 2 \left( C \left( d\left( d\pi \right)(\tilde{u}) h, \tilde{u}\right) + C \left( d\pi(\tilde{u}),  h \right) \right) \cdot h - | d_{1/2} h |^2
	\end{align}
	Notice the similarity with \eqref{innbyniceestimate}. Our goal is now to estimate the terms involving $C$ in an appropriate manner. For example, we have:
	\begin{align}
		| C(d(d\pi)(\tilde{u}) h, \tilde{u}) |	&= \left| \mathcal{R}\left( d(d\pi)(\tilde{u})h \cdot \nabla \tilde{u} \right) - d(d\pi)(\tilde{u}) h \cdot (-\Delta)^{1/2} \tilde{u} \right| \notag \\
									&\leq \left| \mathcal{R}\left( d(d\pi)(\tilde{u})h \cdot \nabla \tilde{u} \right) \right| + \left| h \right| \| (-\Delta)^{1/2} \tilde{u} \|_{L^{\infty}([0,T] \times S^1)},
	\end{align}
	which shows that we merely have to estimate the contribution of the Riesz operator. We know that up to a constant, we have for any $x \in S^1$:
	\begin{align}
	\label{estimateforhfractionalheatnr1}
		&\mathcal{R}\left( d(d\pi)(\tilde{u})h \cdot \nabla \tilde{u} \right)(x) \notag \\
		&\sim - P.V. \int_{S^1} \left( d(d\pi)(\tilde{u}(x)) h(x) \nabla \tilde{u}(x) - d(d\pi)(\tilde{u}(y)) h(y) \nabla \tilde{u}(y) \right) \cot\left( \frac{x-y}{2} \right) dy \notag \\
		&\sim - P.V. \int_{S^1} \frac{d(d\pi)(\tilde{u}(x)) h(x) \nabla \tilde{u}(x) - d(d\pi)(\tilde{u}(y)) h(y) \nabla \tilde{u}(y)}{| x-y |} \cos \left( \frac{x-y}{2} \right) dy \notag \\
		&\lesssim | h(x) | \| u \|_{C^2}+ | d_{1/2} h |(x) \| d(d\pi)(\tilde{u}) \|_{L^{\infty}} \| \nabla u \|_{L^{\infty}},
	\end{align}
	where we used the formula for the distance on the circle $| x-y | = 2 \sin( (x-y)/2)$. The formula follows by using the fractional Leibniz rule adapted appropriately here. By analogous computations using \eqref{alternativeformofcommutator}, we find:
	\begin{equation}
	\label{estimateforhfractionalheatnr2}
		| C(d\pi(\tilde{u}), h)(x) | \lesssim | h(x) | + | d_{1/2} h |(x)
	\end{equation}
	The constants in \eqref{estimateforhfractionalheatnr1} and \eqref{estimateforhfractionalheatnr2} may depend on $\tilde{u}$ and on the target manifold $N$ (via the projection $\pi$ and its derivatives), but they are independent of $h$. To summarise, we have found the following:
	$$\partial_{t} \left( | h |^2 \right) + (-\Delta)^{1/2} \left( | h |^2 \right) \leq \tilde{C}_{\tilde{u}, N} | h | \left( | h(x) | + | d_{1/2} h |(x) \right) - | d_{1/2} h |^2,$$
	where $\tilde{C}_{\tilde{u}, N} > 0$ is a constant depending on $\tilde{u}$ and $N$. By using the arithmetic geometric mean inequality, we may deduce:
	$$\tilde{C}_{\tilde{u}, N} | h | \left( | h(x) | + | d_{1/2} h |(x) \right) - | d_{1/2} h |^2 \leq \tilde{C}_{\tilde{u}, N} | h |^2 + \frac{\tilde{C}_{\tilde{u}, N}}{4 \delta} | h |^2 + \delta | d_{1/2} h |^2 - | d_{1/2} h |^2,$$
	and choosing $\delta = 1$, we find:
	$$\partial_{t} \left( | h |^2 \right) + (-\Delta)^{1/2} \left( | h |^2 \right) \leq \hat{C}_{\tilde{u}, N} | h |^2.$$
	Invoking the maximum principle for the fractional heat flow yields just as in \cite{wettstein} that $h$ must assume its global extremum on $[0,T] \times S^1$ at time $t=0$. However, as $h(0) = 0$ and $| h |^2 \geq 0$, this implies:
	$$| h |^2 = 0 \Rightarrow h = 0,$$
	which finally implies injectivity of $dH(0)$. Therefore, $dH(0)$ is an injective Fredholm operator of index $0$, which shows that it is surjective, thus invertible. Local existence of $W^{1,p}$-solutions to the equation \eqref{exreg:eq005} exist.\\
	
	It remains to verify smoothness of such a solution $u$. This follows by a bootstrap argument similar to the one for $h$, but taking a bit more care. The key observation is that in each step of the bootstrap of $h$, H\"older regularity with sufficiently close $\alpha$ to $1$ is sufficient to obtain the desired estimates, i.e. $\alpha > 1/2$ and thus $p > 4$ suffice. Let us for now take $p>8$ to make the arguments easier, as we shall see below any $p$ is possible anyways. Namely, we have at every fixed time:
	\begin{align}
		\| (-\Delta)^{s/2 - 1/16} C(d\pi(u), u) \|_{L^{p}(S^1)}	&\lesssim \| d\pi(u) \|_{F^{s}_{p,2}(S^1)} \| \nabla u \|_{L^{p}(S^1)} \notag \\
												&\lesssim \| u \|_{C^{0,\alpha}(S^1)} \| u \|_{W^{1,p}(S^1)},
	\end{align}
	if $1 - 1/8 = 7/8 = \alpha > s$ and by integrating in time-direction:
	\begin{equation}
		\| (-\Delta)^{s/2 - 1/16} C(d\pi(u), u) \|_{L^{p}([0,T] \times S^1)} \lesssim \| u \|_{C^{0,\alpha}([0,T] \times S^1)} \| u \|_{W^{1,p}([0,T] \times S^1)}
	\end{equation}
	Arguing as for $h$ before by symmetrisation and \cite{hieber}, this shows:
	$$(-\Delta)^{t/2} u \in W^{1,p}([0,T] \times S^1),$$
	for all $0 \leq t < 3/4$. Next, we verify that $u \in W^{2,p}$ for all $1 < p < \infty$. This follows from the formulation \eqref{exreg:eq005} by rewriting:
	$$(-\Delta)^{1/2} \left( \pi(u) \right)(x) - d\pi(u(x)) (-\Delta)^{1/2} u(x) = P.V. \int_{S^{1}} \frac{\pi(u(x)) - \pi(u(y)) - d\pi(u(x)) (u(x) - u(y))}{| x-y |^2} dy$$
	By using Taylorapproximation, we see for every $j \in \{ 1, \ldots n \}$:
	\begin{align}
		&\pi_j (u(x)) - \pi_j (u(y)) - d\pi_j (u(x)) (u(x) - u(y)) \notag \\
		&= \int_{0}^{1} d\pi_j ((1-t)u(y) + u(x)) (u(x) - u(y)) - d\pi_j (u(x))(u(x) - u(y)) dt \notag \\
		&= \sum_{k=1}^{n} \int_{0}^{1} \left( \partial_k \pi_{j} ((1-t)u(y) + tu(x)) - \partial_{k} \pi_{j}(u(x)) \right) (u_{k}(x) - u_{k}(y)) dy \notag \\
		&= \sum_{k=1}^{n} \sum_{l=1}^{n} \int_{0}^{1} \int_{0}^{1} (t-1) \partial_{kl} \pi_{j} ((s-st) u(y) + (1 + st - s) u(x)) (u_{k}(x) - u_{k}(y)) (u_{l}(x) - u_{l}(y)) ds dt \notag \\
		&=: P_{j}^{kl}(u(x),u(y)) (u_{k}(x) - u_{k}(y)) (u_{l}(x) - u_{l}(y))
	\end{align}
	This is precisely the form we alluded to in Section 4.1.3. Therefore:
	\begin{align}
		&(-\Delta)^{1/2} \left( \pi(u) \right)(x) - d\pi(u(x)) (-\Delta)^{1/2} u(x) \notag \\
		&= \sum_{k,l = 1}^{n} P.V. \int_{S^1} P^{kl}(u(x),u(y)) d_{1/2} u_k (x,y) d_{1/2} u_l (x,y) \frac{dy}{| x-y |}
	\end{align}
	Notice that $P^{jk}$ are bounded and thus the RHS of the flow is actually bounded, since we know $u \in C^{0,\alpha}$ for $\alpha > 1/2$ and:
	\begin{align}
		\left| (-\Delta)^{1/2} \left( \pi(u) \right)(x) - d\pi(u(x)) (-\Delta)^{1/2} u(x) \right|	&= \left| \sum_{k,l = 1}^{n} P.V. \int_{S^1} P^{kl}(u(x),u(y)) d_{1/2} u_k (x,y) d_{1/2} u_l (x,y) \frac{dy}{| x-y |} \right| \notag \\
																	&\lesssim \int_{S^1} \frac{| u(x) - u(y) |^2}{| x-y |^2} dy = | d_{1/2} u |(x)^2 \notag \\
																	&\lesssim \| u \|_{C^{0,\alpha}},
	\end{align}
	where $\alpha > 1/2$. This implies that $u \in W^{1,p}([0,T] \times S^1)$ for all $1 < p < + \infty$, since the RHS of the fractional harmonic gradient flow for $u$ is thus bounded and therefore in all $L^p$-spaces, see \cite{hieber}.\\
	
	Our goal is now to establish higher integrability: We may now differentiate this expression with respect to $x$ to find:
	\begin{align}
		&\frac{d}{dx} \left( (-\Delta)^{1/2} \left( \pi(u) \right)(t,x) - d\pi(u(x)) (-\Delta)^{1/2} u(t,x) \right) \notag \\
		&= \sum_{k,l = 1}^{n} P.V. \int_{S^1} P^{kl}(u(x),u(y)) d_{1/2} u_{k}'(x,y) d_{1/2} u_{l}(x,y) \frac{dy}{| x-y |} \notag \\
		&+ \sum_{k,l = 1}^{n} P.V. \int_{S^1} P^{kl}(u(x),u(y)) d_{1/2} u_{k}(x,y) d_{1/2} u_{l}'(x,y) \frac{dy}{| x-y |} \notag \\
		&+ \sum_{k,l = 1}^{n} P.V. \int_{S^1} \left( dP^{kl}(u(x),u(y)) \begin{pmatrix} u'(x) \\  u'(y) \end{pmatrix} \right) d_{1/2} u_k (x,y) d_{1/2} u_l (x,y) \frac{dy}{| x-y |}
	\end{align}
	It can now he seen, using again H\"older continuity and the previously proven regularity as well as Sobolev embeddings:
	$$\frac{d}{dx} \left( (-\Delta)^{1/2} \left( \pi(u) \right)(t,x) - d\pi(u(x)) (-\Delta)^{1/2} u(t,x) \right) \in L^{p}([0,T] \times S^1).$$
	This now shows:
	\begin{equation}
		u \in W^{2,p}([0,T] \times S^1),
	\end{equation}
	by inserting the regularity $u' \in W^{1,p}$ into the main equation to establish higher regularity in time-direction. Higher order regularity can now be proven by iteration. The result of Proposition \ref{exreg:proplocalreg} therefore follows. \qed \\

	Let us observe that Proposition \ref{exreg:proplocalreg} actually proves existence of solutions to the fractional harmonic gradient flow for sufficiently small times for all closed target manifolds $N$, provided the initial datum is smooth. The next section shall remove the regularity assumption on the boundary data by following \cite{struwe1} as in \cite{wettstein}.
	
	\subsubsection{Global Regularity by Approximation, Existence as a Byproduct}
	
	To prove existence and regularity of solutions in the case of general initial data, we first have to be able to approximate the boundary data sufficiently well by smooth functions, The following result follows precisely as in \cite{wettstein} and the proof is therefore omitted:
	
	\begin{lem}
	\label{approximationlemmauhlenbeckgeneral}
		Assume $N$ is an arbitrary closed manifold. Let $u \in H^{1/2}(S^1;N)$. Then there exists a sequence $u_k \in C^{\infty}(S^1) \cap H^{1/2}(S^1;N)$ such that:
		$$\| u_k - u \|_{H^{1/2}(S^1)} \to 0, \quad k \to \infty.$$
	\end{lem}
	
	The next lemma proven in \cite{wettstein} continues to hold, as its proof relies on general properties of the Triebel-Lizorkin spaces on the unit circle, while the target manifold is irrelevant:
	
	\begin{lem}
	\label{struwelemma3.1general}
		There exist $C >  0$ not depending on $R, u, T$, such that for any smooth $u$ on $[0,T] \times S^1$ and $0 < R < 1$, the following estimate holds for all $x_0 \in S^1$:
		\begin{align}
		\label{struweest01}
			\int_{0}^{T} \int_{B_{\frac{3R}{4}}(x_0)} | (-\Delta)^{1/4} u |^4 dx dt 	&\leq C \sup_{0\leq t \leq T} \int_{B_{R}(x_0)} | (-\Delta)^{1/4} u(t) |^2 dx \notag \\
																&\cdot \left( \int_{0}^{T} \int_{B_{R}(x_0)} | (-\Delta)^{1/2} u |^2 dx dt + \frac{1}{R^2} \int_{0}^{T} \int_{S^1} | (-\Delta)^{1/4}u |^2 dx dt \right),
		\end{align}
		by density the same result applies for all $u \in H^{1}([0,T] \times S^1)$ with bounded $1/2$-Dirichlet energy. Similarily, we have:
		\begin{align}
		\label{struweest02}
			\int_{0}^{T} \int_{S^1} | (-\Delta)^{1/4} u |^4 dx dt 	&\lesssim \sup_{0\leq t \leq T, x \in S^1} \int_{B_{R}(x)} | (-\Delta)^{1/4} u(t) |^2 dx \notag \\
															&\cdot \left( \int_{0}^{T} \int_{S^1} | (-\Delta)^{1/2} u |^2 dx dt + \frac{1}{R^3} \int_{0}^{T} \int_{S^1} | (-\Delta)^{1/4}u |^2 dx dt \right).
		\end{align}
	\end{lem}
	
	Furthermore, due to the orthogonality of the RHS of \eqref{gradflowuvgeneral} with respect to the tangent space of $N$, we also may generalise the following lemmas found in \cite{wettstein}, as the orthogonality is the only property used:
	
	\begin{lem}
	\label{monotone1/2energydecayinglemmageneral}
		Let $u$ be a sufficiently regular solution of the $1/2$-harmonic gradient flow in $N$ as previously defined with $u(0, \cdot) = u_0$ taking values in $N$. Then the following holds for all $T \geq 0$:
		$$\frac{1}{2} \| (-\Delta)^{1/4} u(T) \|_{L^2(S^1)}^{2} \leq \frac{1}{2} \| (-\Delta)^{1/4} u_0 \|_{L^2(S^1)}^{2}$$
		In fact, the energy $T \mapsto \| (-\Delta)^{1/4} u(T) \|_{L^2(S^1)}$ monotonically decreases in $T$.
	\end{lem}
	
	As in \cite{struwe1}, we may introduce for $0 < R < 1$ and $t \in [0,T]$:
	\begin{equation}
		E_{R}(u;x,t) := \frac{1}{2} \int_{B_{R}(x)} | (-\Delta)^{1/4} u(t) |^2 dx,
	\end{equation}
	for the local energy and also:
	\begin{equation}
		\varepsilon(R) = \varepsilon(R;u,T) := \sup_{x \in S^1, t \in [0,T]} E_{R}(u;x,t)
	\end{equation}
	The local energy estimate from \cite{struwe1} and \cite{wettstein} continues to hold by the same proof:
	
	\begin{lem}
	\label{struwelemma3.6general}
		There exists a constant $C > 0$ such that for every $u: [0,T] \times S^1 \to N$ in $H^{1}([0,T] \times S^1) \cap L^{\infty}([0,T]; \dot{H}^{1/2}(S^1))$ solving the half-harmonic flow equation \eqref{gradflowuvgeneral} and satisfying the energy decrease property as in Lemma \ref{monotone1/2energydecayinglemmageneral}, any $0 < R < 1/2$ and $(t,x_0) \in [0,T] \times S^1$, the following estimate holds:
		\begin{align}
		\label{struweest03}
			E_{R}(u;x_0,t) 	&\leq E_{2R}(u;x_0,0) + C \left( \frac{t}{R^2} E(u_0) + \frac{\sqrt{t}}{R} \sqrt{\varepsilon(2R) E(u_0)} \right) \notag \\
						&\leq E_{2R}(u;x_0,0) + C \left( \frac{t}{R^2} + \frac{\sqrt{t}}{R} \right) E(u_0),
		\end{align}
		where $E(u_0) = E_{1/2}(u_0)$. In the second inequality, we used the trivial estimate between the local energy and the global one under the energy decay.
	\end{lem}
	
	Again, the proof is referred to \cite{wettstein}, there are no real differences as the orthogonality of the RHS in \eqref{gradflowuvgeneral} to the tangent space of $N$ removes the non-linearity in the computations.\\
	
	It therefore remains to verify the following results as in \cite{struwe1}:
	
	\begin{lem}
	The following generalisations of the results in \cite{struwe1} hold true:
		\begin{enumerate}
			\item Lemma 3.7 in \cite{struwe1}: There exists $\varepsilon_1 > 0$ such that for any $u \in H^{1}([0,T] \times S^1) \cap L^{\infty}([0,T]; H^{1/2}(S^1))$ solving \eqref{gradflowuvgeneral} with values in $N$ and any $R < 1/2$, there holds:
			\begin{equation}
				\int_{0}^{T} \int_{S^1} | \nabla u |^2 dx dt \leq C E(u_0) \left( 1 + \frac{T}{R^3} \right),
			\end{equation}
			with $C$ independent of $u, T, R$, provided $\varepsilon(R) < \varepsilon_{1}$. Here, $u(0, \cdot) = u_0$ is the initial value.
			
			\item Lemma 3.8, Remark 3.9 in \cite{struwe1}: For any numbers $\varepsilon, \tau, E_0 > 0$ and $R_1 < 1/2$, there is a $\delta > 0$ such that for any $u$, satisfying the conditions as in 1., solving \eqref{gradflowuvgeneral} with values in $N$ and any $I \subset [\tau, T]$ with measure $| I | < \delta$, there holds:
			\begin{equation}
				\int_{I} \int_{S^1} | (-\Delta)^{1/4} u |^2 dx dt < \varepsilon,
			\end{equation}
			provided $\varepsilon(R_1) < \varepsilon_1, E(u_0) \leq E_0$. The same holds with $\tau = 0$, if we consider a sequence $u_n$ associated with converging initial data $u_n(0)$ in $H^{1/2}(S^1)$.
			
			\item Lemma 3.10, Remark 3.11 in \cite{struwe1}: Let $u$ be, in addition to the assumptions in 1., a $C^2([\tau, T] \times S^1)$-solution to \eqref{gradflowuvgeneral}, then, for every $1 \leq p < +\infty$, there exists a $L^{p}([\tau, T] \times S^1)$-bound on $u_t + (-\Delta)^{1/2}u$ with a constant only depending on $E(u_0), \tau, T$ and $R$, provided $\varepsilon( R ) < \varepsilon_1$. Here, $\tau > 0$ in general and $\tau \geq 0$ in case $u_0$ is smooth.
		\end{enumerate}
	\end{lem}
	
	The proof is analogous to \cite{struwe1}, we merely rely on the quadratic estimate for the non-linearity given by:
	$$\left| \sum_{k,l = 1}^{n} P.V. \int_{S^1} a^{kl}(u(x),u(y)) d_{1/2} u_k (x,y) d_{1/2} u_l (x,y) \frac{dy}{| x-y |} \right| \lesssim | d_{1/2} u |^2(x)$$
	Therefore, as in \cite{wettstein}, we refer to \cite{struwe1}, as the proofs are obvious modifications of Struwe's techniques and the previously presented bootstrap procedure for solutions to fractional heat-type equations. Arguing as in Theorem 4.1 in \cite{struwe1}, we may also deduce that for sufficiently small energy at time $t=0$, global existence is ensured. Otherwise, blow-ups may occur.

	\subsection{Convergence of Solutions as $t \to + \infty$}
	
	If we look at the proof of \cite[Theorem 3.4]{wettstein}, it is clear that the arguments immediately generalises to the following Theorem by the same proof:
	
	\begin{thm}
		Let $u \in L^{2}(\R_{+}; H^{1/2}(S^1))$ and $u_t \in L^{2}(\R_{+}; L^{2}(S^1))$ be a solution of the fractional harmonic gradient flow \eqref{gradflowuvgeneral} with values in a closed manifold $N \subset \R^{n}$ and with initial data $u_0 \in H^{1/2}(S^{1};N)$. Assume that:
		$$\| (-\Delta)^{1/4} u(t) \|_{L^2} \leq \| (-\Delta)^{1/4} u_0 \|_{L^2} \leq \varepsilon, \quad \forall t \in \R_{+},$$ 
		for $\varepsilon > 0$ sufficiently small. Then, for a suitably chosen subsequence $t_k \to +\infty$, the sequence of maps $(u(t_k, \cdot))_{k \in \mathbb{N}} \subset H^{1}(S^1;N)$ converges weakly in $H^{1}(S^1)$ to a $1/2$-harmonic map in $N$.
	\end{thm}
	
	We refer to the proof in \cite{wettstein} for details. Again, for sufficiently small $\varepsilon > 0$, we may even deduce that the limit function is a constant map to some point in $N$.

	\begin{appendices}

	\section{Morrey Regularity and Increased Integrability as in \cite{daliopigati}}
	
	In this appendix, we briefly go into some more details of the proof of Lemma \ref{casepe2}. We recall that in the paper, we referred to \cite{daliopigati}, in particular Theorem D.7 and Corollary D.8. Let us expand upon this:\\
	
	We assume that $u$ solves the following equation:
	\begin{equation}
	\label{maineqforappendixc}
		(-\Delta)^{1/2} u = d_{1/2} u \cdot d_{1/2} \left( d\pi^{\perp}(u) \right) + \div_{1/2} \left( \frac{A^{i}_{u}(du, du)(x,y)}{| x-y |^{1/2}} d\pi^{\perp}(u(y))_{ij} \right) + f,
	\end{equation}
	then we notice that:
	$$\int_{S^1} d_{1/2} u \cdot d_{1/2} \left( d \pi(u) \varphi \right) dx = \int_{S^{1}} f \varphi dx,$$
	by arguing as in Section 5.1.2. Therefore, this shows:
	$$d\pi(u) (-\Delta)^{1/2} u = d\pi(u) f$$
	We shall sometimes write $d\pi$ instead of $d\pi(u)$ and $d\pi^{\perp}$ instead of $d\pi^{\perp}(u)$, implying the appropriate functions to be inserted.\\

	If we define $w = u \circ \Pi^{-1}$ using the stereographic projection as in \cite{daliopigati}, this becomes:
	$$d\pi(w) (-\Delta)^{1/2} w = d\pi(w) \tilde{f},$$
	and therefore:
	\begin{equation}
		(-\Delta)^{1/2} w = d\pi(w) \tilde{f} + d\pi^{\perp} (w) (-\Delta)^{1/2}w,
	\end{equation}
	where:
	$$\tilde{f} = \frac{2}{1+x^2} f \circ \Pi^{-1}.$$
	We let:
	\begin{equation}
		v := (-\Delta)^{1/4} w,
	\end{equation}
	and by following precisely the arguments as in \cite{daliopigati} on p.31-32, we find:
	\begin{equation}
	\label{appendixceq01}
		(-\Delta)^{1/4} v = \Omega_{0} v + \Omega_1 v + (-\Delta)^{1/4} \left( d\pi^{\perp} v \right) + 2 (-\Delta)^{1/4} d\pi^{\perp} \cdot d\pi^{\perp} v - T(d\pi^{\perp}, v) + \tilde{f}
	\end{equation}
	Here, $\Omega_0 := d\pi^{\perp} (-\Delta)^{1/4} d\pi^{\perp} - (-\Delta)^{1/4} d\pi^{\perp} d\pi^{\perp}, \Omega_1 := T^{\ast}(d\pi^{\perp}, d\pi)$ and $T$ are the same objects as defined in \cite{daliopigati}, in particular $T$ and $T^{\ast}$ are the following commutators:
	\begin{align}
		T(Q,v) 		&:= (-\Delta)^{1/4} (Qv) + (-\Delta)^{1/4} Q \cdot v - Q (-\Delta)^{1/4} v \\
		T^{\ast}(P,Q) 	&:= (-\Delta)^{1/4} (PQ) - (-\Delta)^{1/4} P \cdot Q - P (-\Delta)^{1/4} Q,
	\end{align}
	satisfying the estimates due to compensation properties of $T, T^{\ast}$:
	\begin{align}
		\| T(Q, v ) \|_{\mathcal{H}^{1}(\R;\R^{m})} 			&\lesssim \| Q \|_{\dot{H}^{1/2}(\R;\R^{m\times m})} \| v \|_{L^2(\R;\R^{m})} \\
		\| T^{\ast}(P,Q) \|_{L^{2,1}(\R, \R^{m \times m})} 	&\lesssim \| P \|_{\dot{H}^{1/2}(\R, \R^{m \times m})} \| Q \|_{\dot{H}^{1/2}(\R, \R^{m \times m})},
	\end{align}
	for $P, Q \in \dot{H}^{1/2}(\R; \R^{m \times m}) \cap L^{\infty}(\R)$ and $v \in L^{2}(\R;\R^m)$. See Appendix C in \cite{daliopigati} for further details. Due to the similar structure of the equation, it is not surprising that the following holds:
	
	\begin{thm}[Theorem D.7, \cite{daliopigati}]
	\label{appendixcthm1}
		The map $v = (-\Delta)^{1/4} w$ has $(-\Delta)^{1/4}(d\pi v), \mathcal{R}(-\Delta)^{1/4}(d\pi^{\perp} v) \in L^{1}(\R)$ and there exists $\alpha > 0$, such that:
		$$\| (-\Delta)^{1/4}(d\pi v ) \|_{L^{1}(B_{r}(x_0))} + \| \mathcal{R} (-\Delta)^{1/4}(d\pi^{\perp} v ) \|_{L^{1}(B_{r}(x_0))} \lesssim r^{\alpha},$$
		for all $r > 0$ and uniformly in $x_0 \in \R$.
	\end{thm}
	
	\begin{proof}
		The change of gauge argument and localisation estimates work equally well in the case of our new equation \eqref{appendixceq01}. Thus, Step 1 carries over word by word. In Step 2, we just need to change the estimate slightly to account for $\tilde{f}$. Namely, we replace $(-\Delta)^{1/4}(Qh)$ in \cite{daliopigati} immediately by $(-\Delta)^{1/4} \tilde{w}_{0}$ with $\tilde{w}_{0}$ being the pullback under the stereographic projection of $w_0$ which solves:
		$$(-\Delta)^{1/2} w_0 = Q\circ \Pi \cdot f,$$
		with $Q$ the gauge from \cite{daliopigati}. The argument proceeds as outlined in section 5 and shows that $(-\Delta)^{1/4} \tilde{w}_{0}$ lies in all $L^{q}(\R)$, for $2 \leq q < \infty$, by Hardy-Littlewood-Sobolev inequality. Therefore, we insert $(-\Delta)^{1/4} \tilde{w}_{0}$, obtaining the same expression as on the bottom of p.32 of \cite{daliopigati} and by using H\"older's inequality to estimate:
		$$\| (-\Delta)^{1/4} \tilde{w}_{0} \|_{L^{2}(B_{r}(x_0))} \lesssim r^{\beta},$$
		for any $\beta \in ]0, 1/2[$ and the estimates on p.33-34 in \cite{daliopigati}, one may deduce completely analogous for some $0 < \gamma < 1/4$:
		$$\| v \|_{L^{2,\infty}(B_{r}(x_{0}))} \lesssim r^{\gamma},$$
		for all $r > 0$ and $x_0$. Thus, Step 2 of the proof of Theorem D.7 still applies.\\
		
		The remainder of the proof of Theorem D.7 in \cite{daliopigati} can now be generalized as well. The application of Adams' embedding is immediate, the $L^{2}$-Morrey decay of $v$ can be obtained by the same trick and due to this, Step 3 holds. Finally, Step 4 and thus the conclusion of the proof of Theorem \ref{appendixcthm1} follow completely analogous by commutator estimates.
	\end{proof}
	
	Looking at Corollary D.8 in \cite{daliopigati} and its proof reveals that the local integrability $(-\Delta)^{1/4} w \in L^{p}_{loc}(\R)$ follows immediately by the same arguments as given there. Therefore, the remainder of the argument in section 5 can be applied and provides the desired gain in integrability.
	
	\end{appendices}

\end{document}